\PassOptionsToPackage{table}{xcolor}
\PassOptionsToPackage{svgnames}{xcolor}
\documentclass[review,onefignum,onetabnum]{siamonline250211}

\usepackage{lipsum}
\usepackage{amsfonts}
\usepackage{graphicx}
\usepackage{epstopdf}

\DeclareMathOperator{\R}{\mathbb{R}}
\DeclareMathOperator{\N}{\mathbb{N}}
\DeclareMathOperator{\softplus}{softplus}

\DeclareMathOperator{\calB}{\mathcal{B}}

\DeclareMathOperator{\Esp}{\mathbb{E}}

\newcommand{\shortminus}{\scalebox{0.5}[1.0]{$-$}}

\def\xnom{x^{\textnormal{nom}}}
\def\fnom{f^{\textnormal{nom}}}
\def\xpi{x^{\textnormal{pi}}}
\def\Vpi{V^{\textnormal{pi}}}

\newtheorem{prop}{Proposition}[section]

\crefname{lemma}{lemma}{lemmas}
\Crefname{lemma}{Lemma}{Lemmas}
\crefname{theorem}{theorem}{theorems}
\Crefname{theorem}{Theorem}{Theorems}
\crefname{prop}{proposition}{propositions}
\Crefname{prop}{Proposition}{Propositions}

\newtheorem{hypo}{Assumption}[section]
\crefname{hypo}{assumption}{assumptions}
\Crefname{hypo}{Assumption}{Assumptions}

\DeclareMathSymbol{\mlq}{\mathord}{operators}{``}
\DeclareMathSymbol{\mrq}{\mathord}{operators}{`'}

\makeatletter
\newcommand*\frob{\mathpalette\bigcdot@{.7}}
\newcommand*\bigcdot@[2]{\mathbin{\vcenter{\hbox{\scalebox{#2}{$\m@th#1\bullet$}}}}}
\makeatother

\usepackage{setspace}
\makeatletter
\newcommand{\labitem}[2]{%
\def\@itemlabel{\textbf{#1}}
\item
\def\@currentlabel{#1}\label{#2}}
\makeatother

\usepackage{csquotes}

\makeatletter
\def\blfootnote{\gdef\@thefnmark{}\@footnotetext}
\makeatother

\usepackage{eurosym}
\usepackage{amsfonts, amsmath, amssymb, dsfont}
\usepackage{mathtools}
\usepackage{algpseudocode}
\usepackage{xcolor} 
\usepackage{multirow,booktabs}
\usepackage[justification=centering]{caption}
\usepackage{subcaption}
\usepackage{xurl}

\ifpdf
  \DeclareGraphicsExtensions{.eps,.pdf,.png,.jpg}
\else
  \DeclareGraphicsExtensions{.eps}
\fi

\usepackage{enumitem}
\setlist[enumerate]{leftmargin=.5in}
\setlist[itemize]{leftmargin=.5in}

\newsiamremark{remark}{Remark}
\newsiamremark{hypothesis}{Hypothesis}
\crefname{hypothesis}{Hypothesis}{Hypotheses}
\newsiamthm{claim}{Claim}
\newsiamremark{fact}{Fact}
\crefname{fact}{Fact}{Facts}

\headers{A Rank-Based Reward between a Principal and a Field of Agents}{C. Alasseur, E. Bayraktar, R. Dumitrescu and Q. Jacquet}

\title{\textbf{A Rank-Based Reward between a Principal and a Field of Agents: Application to Energy Savings}}
\author{Clémence Alasseur\thanks{EDF Lab Saclay, Palaiseau, France (\email{\{clemence.alasseur,quentin.jacquet\}@edf.fr})}
\and Erhan Bayraktar\thanks{Department of Mathematics, University of Michigan, USA (\email{erhan@umich.edu})}
\and Roxana Dumitrescu\thanks{Department of Mathematics, King's College London, United Kingdom (\email{roxana.dumitrescu@kcl.ac.uk})}
\and Quentin Jacquet\footnotemark[1] \thanks{INRIA, CMAP, Ecole Polytechnique, Palaiseau, France}
}

\usepackage{amsopn}

\ifpdf
\hypersetup{
  pdftitle={A Rank-Based Reward between a Principal and a Field of Agents: Application to Energy Savings},
  pdfauthor={C. Alasseur, E. Bayraktar, R. Dumitrescu and Q. Jacquet}
}
\fi

\begin{document}

\maketitle

% REQUIRED
\begin{abstract}
In this paper, we consider the problem of a Principal aiming at designing a reward function for a population of heterogeneous agents. 
We construct an incentive based on the ranking of the agents, so that a competition among the latter is initiated. We place ourselves in the limit setting of mean-field type interactions and prove the existence and uniqueness of the equilibrium distribution for a given reward, for which we can find an explicit representation. Focusing first on the homogeneous setting, we characterize the optimal reward function using a convex reformulation of the problem and provide an interpretation of its behavior. We then show that this characterization still holds for a specific type of heterogeneous population. For the general case, we propose a numerical method which fully exploits the characterization of the mean-field equilibrium. We develop a case study related to the French market of Energy Saving Certificates based on the use of realistic data, which shows that the ranking system allows to achieve the sobriety target imposed by the regulation.
\end{abstract}

% REQUIRED

% REQUIRED

% =================================================
\section{Introduction}

\subsection{Motivation and contributions}
Energy retailers may have incentives to generate energy consumption savings at the scale of their customer portfolio. For example, in France, since 2006, power retailers -- called \textit{Obligés} -- have a target of a certain amount of Energy Saving Certificates\footnote{\url{https://www.powernext.com/french-energy-saving-certificates}} to hold at a predetermined future date (usually 3 or 4 years). If they fail to obtain this number of certificates, then they face financial penalties. Certificates can be acquired either by certifying energy savings at the customer or by buying certificates on the market. 
%Let's note that we are not modeling this arbitrage possibility: buy certificates on the market or reward the customer to provide savings and we only consider the second possibility (of course  this could be an improvement of the model). 
If a retailer holds more certificates than its target at the end of the period, the surplus can be sold on the Energy Saving Certificates market. 
The pluriannual energy savings goal is determined by the regulator, and is function of the cumulative discounted amount of energy saved (thanks to thermal renovation for instance)\footnote{\url{https://www.ecologie.gouv.fr/dispositif-des-certificats-deconomies-denergie}}. Similar mechanisms -- called \textit{White certificates} -- have been implemented in several countries in Europe (Great Britain, Italy or Denmark). This type of mechanisms is not limited to Europe, for example \cite{rosenow2019market} mentions 46 Energy Efficiency Obligation mechanisms across the globe including 24 in the US, 14 in Europe, 4 in Australia.

There is evidence from behavioral economics that energy consumption reductions can be motivated by providing a financial reward and/or information on social norms or comparison to customers, see e.g. see~\cite{Hunt2014} or~\cite{Dolan2015}.
Especially, in~\cite{Dolan2015}, the authors find that social norms reduce consumption by around 6\% (0.2 standard deviations). Secondly, they obtain that large financial rewards for targeted consumption reductions work very well in reducing consumption, with a 8\% reduction (0.35 standard deviations) in energy consumption.
For recent years, electricity providers are aware of this lever to make energy savings, and contracts offering bonus/rewards in compensation of reduction efforts appear, see e.g. the offers of ``SimplyEnergy"\footnote{\url{https://www.simplyenergy.com.au/residential/energy-efficiency/reduce-and-reward}}, ``Octopus Energie"\footnote{\url{https://www.octopusenergy.fr/aide-faq/parrainage-cagnotte/fonctionnement-cagnotte-octopus}} or ``OhmConnect"\footnote{\url{https://www.ohmconnect.com/}}. The interest of this kind of solutions has been reinforced during the 2022 energy crisis  where many countries intend to diminish their global energy consumption\footnote{\url{https://www.politico.eu/article/eu-countries-save-energy-winter/}} because of gas and power shortage.

In light of the conclusions of~\cite{Dolan2015}, the present work is motivated by the ambition of constructing a \textit{new type of rewards} which are \textit{rank-dependent}. This new mechanism design of contracts  is based on the comparison between consumers in terms of their reduction efforts. On the mathematical side, our problem falls into the category of Stackelberg mean-field games with rank-based feature, in which a  principal (e.g.,~social planner, price-maker) is interested in the design of rank-based game rewards to optimize a certain 
criterion. This type of games has been quite rarely explored in the literature, being recently introduced in a series of papers, see e.g.~\cite{Bayraktar_2016, Bayraktar_2018, Bayraktar_2019}. From a mathematical perspective, to be able to conduct a rigorous study of our model on the market of Energy Savings, we develop several extensions of  the theoretical results  obtained in the latter papers. Finally, numerical algorithms for this type of games - especially in the case of heterogeneous populations, for which explicit characterization of the equilibrium is not available -  are completely missing in the literature, and we also aim at contributing in this direction.

\vspace{3mm}

% --------------------------------
We summarize below the main contributions of this paper:
\begin{itemize}
\item We design a \textit{new type of contract} which consists in offering a monetary reward to each consumer based on the rank of their consumption. In our context, the rank measures the reduction effort of a consumer compared with the rest of the population (a rank $r\in[0,1]$ indicates that the consumer is among the 100 $\times$ $r$ percent of the population with the highest consumption reduction). This new mechanism initiates a competition between similar consumers to be the best energy saver and unites the incentive potential of rankings with a financial reward. The problem writes as a \textit{Stackelberg mean-field game with rank-based interaction}.
\item We apply our model to the French market of Energy Saving Certificates using realistic data~(\Cref{sec::num_res}). Our main findings in this part are that this new type of contract based on a ranking system allows to achieve the sobriety target imposed by the regulation. In addition, we reach a simple form for the contract which can be easily described to customers. Both this simple form and the level of reward obtained, which is in line with existing offers, confirm that the contract we designed is realistic and can be readily implemented in practice.
\item From a mathematical point of view, we extend the results of~\cite{Bayraktar_2019} in several directions.  In particular, we provide new theoretical results in the case of a homogeneous population for reward functions which depend not only on the rank, as in~\cite{Bayraktar_2019}, but also on the cumulated consumption. We get semi-explicit representation of the equilibria and explicit characterization of the contract (\Cref{theorem:q_mu} and ~\Cref{prop::optimal_reward}). Furthermore, we are able to extend the previous results to the case of general convex cost functions and a specific type of heterogeneous population  (\Cref{prop::optimum_heterogeneous}). Finding such explicit expressions is rare in the literature, and is only possible by imposing a specific dynamics (as in~\cite{Elie_2019} and~\cite{Carmona_2021}).
Besides, we consider several extensions suitable to our context. First, we show that, for the class of reward functions considered here, the addition of common-noise in the consumption process only shifts the equilibrium distribution by a (random) constant. We also focus on time-dependent costs of effort for the agents, reflecting the collective awareness of agents on the energy reduction's necessity. We provide some invariance results, which show that the use of more sophisticated reward (a function that jointly depends on the rank and the consumption of the agent) is, at the equilibrium, equivalent to a reward that belongs to the class of purely rank-based rewards. 
\item For the realism of the application, we have also to consider the case of a general heterogeneous population, which cannot be solved explicitly anymore. We then introduce a numerical algorithm which exhibits a fast convergence. In the particular  setting of a homogeneous population,  we numerically show that it successfully finds the optimal reward obtained in explicit form.
\end{itemize}

% --------------------------
\subsection{Related literature}

For a given reward function provided by the retailer, the competition between agents is modeled by a mean-field game. These games have been introduced simultaneously by Lasry and Lions~\cite{lasry2006jeux-I, lasry2006jeux-II,lasry2007mean} and Huang, Caines and Malhame~\cite{huang2006, huang2007}. They refer to the study of differential games involving a large number of indistinguishable agents which interact through their empirical distribution. By looking at the limit case where a continuum of agents is involved, each of them asymptotically negligible, mean-field games provide efficient ways to compute approximations of Nash equilibria for stochastic games with large number of players (games which are otherwise rarely tractable). Among various techniques, the problem is often solved by a fixed-point method involving both a Hamilton-Jacobi-Bellman equation -- characterizing the agents best response to a given  population distribution -- and a Fokker-Planck equation. Existence and uniqueness of a mean-field equilibrium are then analyzed through this system of coupled partial differential equations, see e.g.~\cite{Cardaliaguet_2015,Benamou_2017}.

The design of a reward/incentive by the retailer is then modeled as a \textit{Principal-Agent} problem, see e.g. the works of Sannikov~\cite{Sannikov_2008} and Capponi, Cvitani\'c and Yolcu~\cite{Capponi_2012} in continuous-time settings. In such problems, the Principal (retailer) aims at designing a monetary reward that is offered to the agent, depending on the quantity of work achieved by the latter. In energy management, A\"id, Possama\"i and Touzi~\cite{Aid_2022} introduces an incentive mechanism to control both the average fo the instantaneous consumption and the volatility of the agents instantaneous consumption. The additional difficulty in our context is the presence of a continuum of agents, and the interaction between them which is expressed in terms of a mean-field game. Such extensions of the Principal-Agent problem have been considered by Elie, Mastolia and Possama\"i~\cite{Elie_2019} and by \cite{elie2021mean} for demand response contracts in electricity markets -- where an explicit contract has been found for a specific class of dynamics (encompassing the linear-quadratic setting) -- and by Carmona and Wang~\cite{Carmona_2021} -- focusing on the linear-quadratic setting and finite-state spaces. Recent works proposed new formulations such as in~\cite{chiusolo2024new} or a model with a more complex hierarchy in~\cite{ hubert2023continuous}. Shrivats, Firoozi and Jaimungal~\cite{Shrivats_2021} introduce a Principal-Agent formulation to study the interaction between a regulator and a field of providers in the market of Renewable Energy Certificate (REC). 

The case of mean-field games with rank-based interactions have been first introduced in~\cite{Bayraktar_2016}, where  results of existence and uniqueness of the mean-field Nash equilibrium are provided for a general class of rewards. Extensions to principal-agent problem are then studied in~\cite{Bayraktar_2018,Bayraktar_2019}, deriving explicit expressions of optimal contract for several principal's objectives (profit/effort/rank-performance maximization/distribution target). These papers are the closest related to our paper, and our theoretical contributions with respect to them are presented above.

Finally, on the numerical side,  Campbell et al. introduce in~\cite{Campbell_2021} deep learning algorithms to solve principal-agent mean field games under heterogeneity of agent types. Here, we propose an alternative method, which takes advantage of the specific structure of the problem (explicit solution of the underlying mean-field game and common rank-based reward across the sub-populations) to lower the numerical complexity and derive efficient computational methods.

The rest of the paper is organized as follows: in~\Cref{sec::model}, we first define the model and characterize the equilibrium for the mean-field game between the agents. In~\Cref{sec::reward_optim}, we propose a numerical approach to solve the problem in the heterogeneous setting, for which the convex reformulation seems not extendable. In~\Cref{sec::num_res}, we apply the results to the French market of Energy Savings Certificates, and finally in~\Cref{sec::extensions}, we tackle some extensions that naturally arise in the context of the application.

\emph{The proofs of the main results are given in the
Appendix.}

% =======================================================================

% -----------

\section{Model}\label{sec::model}

\subsection{Notation and assumptions}

In the sequel, we denote by $\mathcal{P}(\R)$ the set of distributions defined on $\R$ and by $\mathcal{P}^+(\R)$ the set of distributions having strictly positive density. Moreover, for any $\mu\in\mathcal{P}(\R)$, $F_\mu$ refers to the cumulative distribution function (cdf) of $\mu$, and when it exits, $f_\mu$ (resp. $q_\mu$) refers to the probability density function (pdf) (resp. the quantile function) of $\mu$. Moreover, we write $X\sim \mu$ when $X$ is distributed according to $\mu\in\mathcal{P}(\R)$. The normal distribution centered in $m$ with standard deviation $\sigma$ is denoted by $\mathcal{N}(m,\sigma)$ and its pdf is denoted by $x \mapsto \varphi(x;m,\sigma)$.
Let us successively introduce the different players involved in the Stackelberg game:
\begin{itemize}
    \item[$\diamond$] \textbf{Consumers.} We consider a \emph{heterogeneous} population of consumers, and we suppose that a clustering algorithm can be applied as a preprocessing step in order to split the population into $K$ sub-populations (or clusters), each of them composed of similar customers. Each cluster $k\in[K]:=\{1,\hdots,K\}$ represents a proportion $\rho_k$ of the overall population and corresponds to a given class of customers, categorized for example according to their usages, their heating system or the household composition. Here, we directly tackle mean-field interactions between the agents:
\begin{hypo}
    We assume that each sub-population is composed of an infinite number of indistinguishable agents, represented by a single consumer (\emph{representative} agent).
\end{hypo}

\textit{Energy consumption.}
In our setting, the incentive that initiates the competition between the agents is based on the \emph{cumulated} consumption of each representative agent over the time period $0$ to $T$. Therefore, we aim at representing by a stochastic process the dynamics of the forecast of the cumulated energy consumption, in contrast with papers where the instantaneous energy consumption is studied (see e.g.~A\"id, Possama\"i and Touzi~\cite{Aid_2022}).
To this purpose, let $(\Omega, \mathbb{F}, \mathbb{P})$ be a complete filtered probability space, which supports a family of $K$ independent Brownian motions $\{W_k\}_{1 \leq k \leq K}$.
For each $k$, the cumulated consumption $X_k^0(T)$ reads as follows:
\begin{align}
X_{k}^{0}(T)=\xnom_k+\int_0^T \sigma_k dW_{k}(s).
\end{align}
We define the \textit{forecast of the cumulated consumption} at time $t \in [0,T]$ by:
\begin{align}
X_{k}^{0}(t)=\mathbb{E}[X_{k}^{0}(T)|\mathcal{F}_t]=\xnom_k+\int_0^t \sigma_k dW_{k}(s).
\end{align}
We now introduce the set $\mathbb{A}$ of progressively measurable processes $a$ satisfying the integrability condition $\mathbb{E} \int_0^T |a(s)|ds<\infty$. 
For a given control $a_k \in \mathbb{A}$ (which can be viewed in our setting as the consumer's effort to reduce his electricity consumption), we introduce the \textit{forecast of the terminal under effort $a_k$} which is given at time $t \in [0,T]$ by:
\begin{equation}\label{eq::dynamics}
   X_k^a(t) = \xnom_k+\int_0^t \sigma_k dW_{k}(s) + \int_0^t a_k(s) ds.
\end{equation}
We can observe that here $\xnom_k$ represents the forecast at time $0$ of the cumulative consumption under zero effort and is called \textit{nominal} energy consumption.
In other words, the process $X_k^a$ (representing the forecast of the cumulated consumption under effort $a$) deviates from the process $X_k^0$ (representing the forecast of the cumulated consumption under a zero effort) according to the cumulative effort made since the beginning of the time period.

 We define by $\fnom_k$ the  p.d.f. of $X^{0}_{k}(T)$ : 
\begin{equation}\fnom_k(x) := \varphi\left(x\,;\,\xnom_k,\sigma_k\sqrt{T}\right)\enspace.\end{equation}

Under effort, we denote by $\mu_k$ the p.d.f. of $X^a_k(T)$ and the corresponding mean cumulated (over time) consumption over the agents of cluster $k$ by $$
m_{\mu_k} = \int_0^1 q_{\mu_k}(r)dr.
$$
For an equilibrium $(\mu_k)_{k\in[K]}$, the mean consumption of the \emph{overall} population is then $m_\mu = \sum_{k\in[K]}\rho_k m_{\mu_k}.$

Note that we do not explicitly impose bounds on the process $X_k$ -- typically non-negativity assumption -- but this will be naturally enforced by the cost of effort and the volatility parameter $\sigma_k$ so that the probability of negative consumption will be negligible. 

\item[$\diamond$]\textbf{Retailer.} In this model, an electricity provider, incentivised by a regulation agency, aims at designing a reward function based on the \emph{terminal ranking} of the agents in order to lower the global consumption of the customers: considering that the cumulated  consumption of the agents in the $k$th population, i.e. $X^a_k(T)$, is distributed according to $\mu_k$, the ranking $r$ of a player consuming the quantity $x$, is  measured by the fraction of agents consuming less than $x$, i.e., $r = F_\mu(x)$, where $F_\mu$ denotes the cumulative distribution function on $\mu$ (so that the worst performer/the highest consumption has rank one and the top performer has rank 0).

A reward function in our context is then a continuous real-valued function $\R \times [0,1] \ni (x,r) \mapsto R(x,r)$ that depends both on the cumulated consumption $x$ and the terminal ranking $r$. We consider only rewards that are non-increasing in both arguments, to favor low ranks. For any $\mu\in\mathcal{P}(\R)$, we write $R_\mu(x) = R(x,F_\mu(x))$ and when $R(x,r)$ is independent of $x$, we say that the reward is \emph{purely rank-based}. In the sequel, we will consider the following decomposition assumption:

\begin{hypo}\label{hypo::reward} Each sub-population $k\in[K]$ receives a reward $R_k$ has the form
\begin{equation}R_k\left(x, r\right) = B_k(r) - px\enspace,\end{equation}
where $p\in\R_+$ and $B_k \in \calB$ with $\calB$ the set of purely rank-based (decreasing) functions.
We then call $R$ the \emph{total reward} and its rank-dependent part $B_k$ the \emph{additional reward} (financial ``bonus" for the consumer).
\end{hypo}

In the energy context, the second member ``$-p x$" represents the classic invoice of the consumer, where $p$ is the price to consume one unit of energy (e.g. in \euro/kWh). Here, this simple pricing strategy can be viewed as a regulated price (as this is the case in France for example\footnote{``Tarif réglementé de vente" (TRV)}). The invoice is embedded in the reward function since it acts as a natural incentive to reduce the consumption. The first member $B_k$ is then the additional financial reward offered to consumers based on their terminal ranking.

In the modeling of energy consumption, a common-noise is often added (it can represent for example the outdoor temperature). However, we show that the insertion of such a noise only shifts the consumption distribution (by a random constant). This result was already mentioned for translation invariant functions (such as purely rank-based rewards), and we extend in~\Cref{sec::extensions} this property to the more general class of reward functions satisfying~\Cref{hypo::reward}.

\begin{hypo}[Fair reward mechanism]
\label{hypo::identical_reward}~
\begin{enumerate}[label=(\roman*)]
\item Each cluster is \emph{independent}: the rank of an agent of cluster $k\in[K]$ is only determined by the distribution of the cluster $k$.
\item The same \emph{unitary} bonus is proposed to each cluster, i.e., $B_k(r) = \xnom_k \beta(r)$ for all $k\in[K]$.
\end{enumerate}
\end{hypo}
Assumption~\ref{hypo::identical_reward} imposes that the sub-populations evolve \emph{separately}, but are linked through a common reward function. This assumption is taken for the sake of a fair reward mechanism: on one hand, consumers only compete with similar agents, i.e., with agents having the same characteristics (type of heating,  household composition, ...) and on the other hand, the shape of the reward should be identical for each the sub-population to prevent from favoring one cluster compared to another. The function $\beta$ is then the unitary bonus received by every customer (in \euro/kWh). This specific structure of rewards  -- linking the populations through a common unitary bonus -- aims to take into account that a larger consumer may need a higher reward to initiate a change of behavior, an effort.

\item[$\diamond$] \textbf{Regulator's incentive.} 
We denote by $\kappa : \R \to \R$ the mean selling cost function (including the regulator's  incentive) that depends on the mean consumption $m_\mu$ of the overall population. The cost function is usually defined as a function of the total consumption, but here we directly introduce proportions $\rho_k$  for the sub-populations instead of the absolute number of customers. The retailer has to then decide between the income generated by the sale of electricity ($pm_\mu$) and its cost ($k(m_\mu)$), see~\Cref{sec::fixed_price} for more details.

\begin{hypo}\label{hypo::saving_function} The function $\kappa:\R \to \R$ is increasing, convex and differentiable. Moreover, denoting by $\xpi_k$ the consumption under
the incentive $R_k(x) = -px$,  $\kappa'(0)< p < \kappa'\left(\sum_{k\in[K]}\xpi_k\right)$.
\end{hypo}
\Cref{hypo::saving_function} is natural in the context of our application. In practice, the selling cost function is defined as $\kappa: m \mapsto c_p(m) + s (m)\enspace,$ where 
\begin{itemize}\setlength\itemsep{-0.2em}
    \item[$\diamond$] $s(\cdot)$ denotes the \emph{valuation function} for the retailer, i.e., the penalty imposed by the regulator to favor a reduction in consumption,
    \item[$\diamond$] $c_p(\cdot)$ denotes the cost function, induced by the supply of energy (or production of energy if the retailer is also the producer).
\end{itemize}
We assume here that the marginal cost $\kappa'(\cdot)$ is lower than the marginal price $p$ at $0$ -- meaning that it is always profitable to sell a positive quantity of energy -- and conversely we assume that the marginal cost $\kappa'(\cdot)$ is greater than the marginal price $p$ at $\xpi$ -- meaning that it is not profitable to sell more electricity with the additional reward than without (the regulator penalizes an overconsumption).
The convexity of the $\kappa$ is a strong assumption, but usually assumed for energy markets:
\begin{itemize}
    \item[$\diamond$] The penalty function $s$ is increasing and convex, since the regulator aims at encouraging consumption reduction by strongly penalizing huge consumption levels. The cause of this function will be described in more details in the numerical section for the energy savings case; see~\Cref{fig::penalty}. 
    \item[$\diamond$] Moreover, the retailer's aggregated cost function is often considered as increasing and convex, due to a decreasing return to scale, see e.g.~\cite{Alekseeva_2019,Alasseur2020}: the mechanism of day-ahead markets favors the “cheapest” (lowest marginal cost) power plants as the cheapest resource will participate to the electricity generation first, followed by the second cheapest option, and so on, until the demand is satisfied. In the case of non-convex aggregated cost, the convex hull of the aggregated cost function is often considered, see e.g.~\cite{Schiro_2016}.
\end{itemize}
\end{itemize}

\begin{figure}[!ht]
    \centering
    \includegraphics[width=0.70\linewidth,clip=true,trim=0cm 2cm 0cm .3cm]{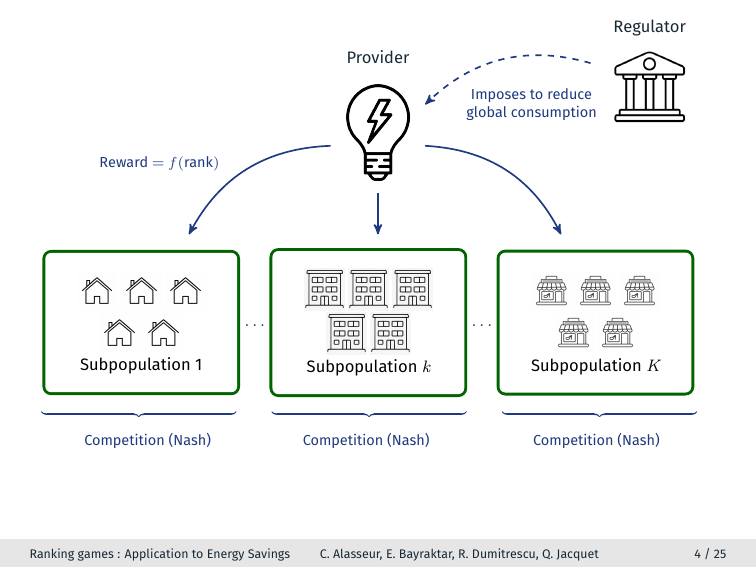}
    \caption{Relation between the Principal (provider) and the sub-populations composed of an infinite number of Agents (consumers).}
    \label{fig:schema_ranking}
\end{figure}

\Cref{fig:schema_ranking} outlines the Principal-Agent relation between the retailer and the field of consumers. We then first focus on the competition among the agents before studying the principal problem.

% --------------------------------------------
\subsection{Mean-field game between agents}
\emph{
In all this section, let us focus on a given cluster $k\in[K]$, as there is no interaction between clusters. We suppose here that the reward $R_k(x,r)$ is given. 
}

\paragraph{Utility function}
An agent of cluster $k$ is able to produce an effort $a_k$ to reduce its consumption, but has to pay as a counter-part the quadratic cost $c_k\, a^2_k(t)$ with $c_k>0$ a given positive constant. The convexity of the effort cost is natural in the context of our application. In particular, this cost either corresponds to the purchase of new equipment that is more efficient than the older one (new heating installation, isolation, ...) or corresponds to a change in the consumption pattern (sobriety). In the latter case, the convexity illustrates that small efforts (as for e.g.  switching off the
light when leaving a room) are easy to make  while large consumption reduction (as for e.g.  reducing heating or air conditioning ) are more demanding. It is also possible to consider a more general convex cost, which is in non-quadratic form, since it would still lead to a tractable agent problem. However, quadratic costs are often considered in order to obtain explicit expression of the optimum, see e.g.~\cite{Aid_2022} in the electricity context. In exchange of the effort, the consumer receives the reward $R_k(x,r)$, depending on his rank $r=F_{\mu_k}(x)$ within the sub-population, where $\mu_k$ is the $k$-sub-population distribution. His objective is then: 
\begin{equation}\label{eq::consumer_problem}\tag{\mbox{$P^{\text{cons}}$}}
    V_k(R_k,\mu_k) := \sup_{a \in \mathbb{A}} \mathbb{E} \left[ R_{k,\mu_k}(X^a_{k}(T)) - \int_0^T c_k a^2_k(t) dt \right]\enspace.
\end{equation}
The quantity $V_k(R_k,\mu_k)$ represents the \emph{optimal expected utility} of an agent of class $k$, for a given provider's reward and population distribution.

% -----
%\subsubsection{Previous results}
%The utility maximization problem~\eqref{eq::consumer_problem} has been extensively studied in~\cite{Bayraktar_2019}, and we recall here some results. 
\paragraph{Previous results} We present below some results which will be used throughout the paper. The first result gives the explicit solution of the agent's best response to a population distribution $\tilde{\mu}_k$:
\begin{prop}[Characterization of the best response]\label{prop::best_response}
Given a bounded total reward function $R_k$ satisfying~\Cref{hypo::reward} and $\tilde{\mu}_k\in\mathcal{P}(\R)$,  let
\begin{equation}
\gamma_k(\tilde{\mu}) = \int_{\R} \fnom_k(x) \exp\left(\frac{R_{k,\tilde{\mu}}(x)}{2c_k\sigma_k^2}\right)dx\quad (<\infty)\enspace.
\end{equation}
Then,
the optimal terminal distribution $\mu_k^*$ of a player from cluster $k$ admits a p.d.f. defined as
\begin{equation}\label{eq::Phi_operator}f_{\mu_k^*}(x) = \frac{1}{\gamma(\tilde{\mu}_k)}\fnom_k(x)\exp\left(\frac{R_{k,\tilde{\mu}_k}(x)}{2c_k\sigma_k^2}\right)\enspace,\end{equation}
and the optimal value is then
$V_k(R_k,\tilde{\mu}_k) = 2c_k\sigma_k^2 \ln \gamma_k(\tilde{\mu}_k)\enspace.
$%\end{equation}
%Finally,
%the optimal effort is given by
%\begin{equation}
%a_k^* = \frac{1}{2c_k}\partial_x V_k(R,\mu_k)
%\end{equation}
\end{prop}
The above result corresponds to \cite[Proposition 2.1]{Bayraktar_2019} and is obtained using the Schrödinger bridge approach, see~\cite{Chen_2015} for connections with optimal transport theory. The consumption process $X_k$ under the optimal effort  then satisfies the equation 
$$dX_k(t) = a_k(t,X_k(t);\mu^*_k)dt + \sigma_kdW_k(t),$$
%, see~\cite[Section 2.2]{Bayraktar_201}, 
where the optimal effort $a_k(\cdot,\cdot;\mu^*_k)$ is defined as
\begin{equation}\label{eq::optimal_effort}
    a_k(t,x,\mu) = \sigma_k^2 \partial_x \ln u_k(t,x,\mu)
\end{equation}
with $$
u_k(t,x,\mu) = \Esp\left[\exp\left(\frac{1}{2c_k\sigma_k^2}R_{k,\mu}(x+\sigma_k\sqrt{T-t}Z\right)\right],\;\; Z\sim \mathcal{N}(0,1)\enspace.
$$
We now introduce the notion of mean-field Nash equilibrium.
\begin{definition}[Mean-field Nash equilibrium]\label{def::Phi}
We say that $\mu_k\in\mathcal{P}(\mathbb{R})$ is an \emph{equilibrium} (terminal distribution) if it is a fixed-point of the mapping $\Phi_k\;: \;\tilde{\mu}_k \mapsto \mu_k^*$, with $\mu_k^*$ given by the solution of the equation \eqref{eq::Phi_operator}.
\end{definition}
The existence of such an equilibrium has been proved in the general setting using Schauder’s fixed point theorem (see~\cite{Bayraktar_2016}). We give below a characterization of this equilibrium distribution, as well as an explicit expression for purely rank-based rewards:
\begin{prop}[Characterization of the equilibrium distribution]\label{lemma:condition_quantile}
Given a bounded total reward function $R_k:\R\times [0,1]\to\R$, the distribution $\mu_k\in\mathcal{P}(\R)$ is an equilibrium terminal distribution for cluster $k$ if and only if its quantile function $q_{\mu_k}$ satisfies 
\begin{equation}\label{eq::charac_eq}
N\left(\frac{q_{\mu_k}(r)-\xnom_k}{\sigma_k\sqrt{T}}\right) = \frac{\int_0^r \exp\left(-\frac{R_{k,\mu_k}(q_{\mu_k}(z))}{2c_k\sigma_k^2}\right) dz}{\int_0^1 \exp\left(-\frac{R_{k,\mu_k}(q_{\mu_k}(z))}{2c_k\sigma_k^2}\right) dz}\enspace,
\end{equation}
where $N$ is the standard normal c.d.f.
In the specific case of a purely rank-based reward, we obtain that the equilibrium $\nu_k$ is unique and the quantile is given by
\begin{equation}\label{eq:q_nu}
q_{\nu_k}(r) = \xnom_k + \sigma_k \sqrt{T} N^{-1}\left(\frac{\int_0^r \exp\left(-\frac{B_k(z)}{2c_k\sigma_k^2}\right) dz}{\int_0^1 \exp\left(-\frac{B_k(z)}{2c_k\sigma_k^2}\right) dz}\right)\enspace.
\end{equation}
\end{prop}

\paragraph{New results} The above result is provided in~\cite[Theorem~3.2]{Bayraktar_2019}, and below we extend the explicit characterization to the more general case of reward maps $R$, which not only depend on the rank, but also have a linear dependence on $x$. 
\begin{theorem}[Explicit characterization for non purely rank-based rewards]\label{theorem:q_mu}
Suppose the reward is of the form defined in~Assumption~\ref{hypo::reward}. Then, the equilibrium $\mu_k$ is unique, and it satisfies
\begin{equation}\label{eq:q_mu}
q_{\mu_k}(r) = q_{\nu_k}(r) - \frac{pT}{2c_k}\enspace,
\end{equation}
where $\nu_k$ is the (unique) equilibrium distribution for the specific case $p = 0$ (purely rank-based reward), defined in~\eqref{eq:q_nu}.
\end{theorem}

\Cref{theorem:q_mu} shows that the addition of a linear part in the consumption acts as a shift on the probability density function.
%\begin{remark}
We emphasize that our uniqueness result of the equilibrium $\mu$ generalizes the one established in \cite{Bayraktar_2019}, the latter being obtained under the additional assumptions that the map $r\mapsto R_k(x,r)$ is convex and $r\mapsto \partial_x R_k(x,r)$ is non decreasing. Instead, we assume a linear dependence on the consumption for the reward, but no convexity requirement is made on its purely rank-based component $B$.
%\end{remark}

\begin{corollary}[Equilibrium without additional reward]\label{corol::pi}
For $R_k(x,r)=-px$, the equilibrium follows the normal distribution $\mathcal{N}\left(\xpi_k,\sigma_k\sqrt{T}\right)$, where $\xpi_k = \xnom_k - \frac{pT}{2c_k}$ is the consumption under the natural incentive associated with the price $p$. Moreover, the optimal consumer's utility is
\begin{equation}\label{eq::value_consumer_noB}
    V_k(R,\mu_k) = \Vpi_{k} := -p \xpi_k - \frac{p^2T}{4c_k}\enspace.
\end{equation}
\end{corollary}
\begin{proof}
    For  $B_k\equiv 0$, Eq.~\eqref{eq:q_nu} gives us $q_{\nu_k}(r) = \xnom_k + \sigma_k \sqrt{T} N^{-1}(r)$, therefore $\nu_k\sim \mathcal{N}(\xnom,\sigma_k\sqrt{T})$. We then obtain by~\Cref{theorem:q_mu} the definition of the equilibrium $\mu_k$. Finally, using~\Cref{prop::convo}, we get
    $$
    2c_k\sigma_k^2 \ln \gamma_k(\tilde{\mu}_k) = \ln\left(\int_{\R} \fnom_k(x) \exp\left(\tfrac{-px}{2c_k\sigma^2_k}\right)dx\right)
    =- p\xnom_k + \frac{p^2T}{4c_k} \enspace .
    $$
\end{proof}

\Cref{corol::pi} shows that the price of electricity constitutes a natural incentive, as the consumer already makes an effort to reduce his consumption from $\xnom$ to $\xpi$. However, it induces a disutility for consumers ($\Vpi\leq 0$). An increase of the price would lead to a supplementary consumption reduction but would decrease further the utility of the agents, and is therefore a non-desirable energy saving strategy.

% ---------------------------------------------
\subsection{The Principal's problem}\label{sec::fixed_price}
%\todo{say that new}
\emph{In this section, we suppose that~\Cref{hypo::reward} is satisfied. Therefore, the equilibrium distribution is unique and is defined by~\eqref{eq:q_mu}.}

For a given $k$, we denote by $\epsilon_k$ the mapping which associates to the total reward function the corresponding equilibrium distribution, i.e. $\epsilon_k(R_k)=\mu_k,$ where $\mu_k$ satisfies~\eqref{eq:q_mu}.
The problem of the retailer can then be written as
\begin{equation}\label{eq::master_problem}\tag{\mbox{$P^{\text{ret}}$}}
\pi^*:=\max_{\beta\in \calB}  \left\{p m_{\mu} - \kappa(m_\mu)- \sum_{k\in[K]}\rho_k\xnom_k\int_{0}^{1} \beta(r)dr\; \left| \; \begin{aligned}
    &R_k(x,r) = \xnom_k \beta(r) - px\\
    &\mu_k = \epsilon_k(R_k)\\
    &V_k(R_k,\mu_k) \geq \Vpi_k +\tau \xnom_k
\end{aligned} \right.\right\}
\end{equation}
where $\kappa(\cdot)$ denotes the mean selling cost function
and $m_\mu$ is the mean consumption at the equilibrium $\mu$. The optimal objective $\pi^*$ then corresponds to the profit per agent (mean over the population) made on the interval $[0,T]$ (in \euro).
The inequality constraint on the utility ensures that consumers ``play the game", as it procures a strictly better utility than without additional reward.  Classically, $\tau = 0$, meaning that the effort achieved by consumers in order to save energy is compensated (in mean) by the reward offered by the retailer. Observe that with $\tau=0$, some agents may have a negative reward, which is not always desirable. Therefore, for practical issue and acceptability, we allow for a positive $\tau$ to take into account switching costs that appear when it comes to subscribing to a reward mechanism, see e.g.~\cite{Magnani_2023}. Here, we choose the additional term ``$+\tau\xnom$" to be proportional to the nominal consumption, supposing that the implicit switching costs that characterize the consumers are proportional to their baseline consumption, i.e., a huge (resp. small) consumer will change if the difference with its baseline option is huge (resp. small).

In the case of a homogeneous population  and \textit{linear dependence of the objective function with respect to the equilibrium distribution}, the results are obtained in \cite{Bayraktar_2019}. We extend them here to the more general case of \textit{convex nonlinear dependencies}.
%In \cite{Bayraktar_2019} and for homogeneous population, the results are obtained in the case of linear dependence of the objective function with respect to the equilibrium distribution, and they are extended here to convex nonlinear dependencies. 

% -----
\subsubsection{Homogeneous population}
We consider in this section the specific case where there is a unique cluster of customers (homogeneous population). Therefore, we omit the dependence in $k$.
Using~\Cref{prop::map_reward} (given in the Appendix), Problem~\eqref{eq::master_problem} can be reformulated as a constrained minimization problem on the distribution space:
\begin{prop}Let us consider the following minimization problem
\begin{equation}\label{eq:reformulated_problem}
\begin{aligned}
\min_{\mu\in \mathcal{P}^+(\R)} &\quad \kappa\left(\int_{\R} y f_\mu(y) dy\right) + 2c\sigma^2\int_{\R} \ln\left(\frac{f_\mu(y)}{\fnom(y)}\right)f_\mu(y)dy\\
\text{s.t.} &\quad  \int_{\R} f_\mu(y) dy = 1\\
&\quad y\mapsto\ln\left(\frac{f_\mu(y)}{\fnom(y)}\right) + \frac{p}{2c\sigma^2}y \text{\quad decreasing}
\end{aligned}\enspace.
\end{equation}
Then, the reward $B_{\mu^*}\in\calB$, constructed from an optimal distribution $\mu^*\in\mathcal{P}^+(\R)$ of~\eqref{eq:reformulated_problem} as 
\begin{equation}\label{eq::optimal_Bstar}
    B_{\mu^*}(r) = \Vpi+\tau \xnom + 2c\sigma^2 \ln\left(\frac{f_{\mu^*}(q_{\mu^*}(r))}{\fnom(q_{\mu^*}(r))}\right)+ p q_{\mu^*}(r)
\end{equation}
is optimal for problem~\eqref{eq::master_problem}.
\end{prop}
\begin{proof}
    From~\Cref{prop::map_reward}, $B_{\mu}$ defined in~\eqref{eq::optimal_Bstar} is the reward that achieves a given equilibrium distribution $\mu$ with the lowest cost while satisfying the utility condition in~\eqref{eq::master_problem} (since $V(R,\mu) = (1+\tau)\Vpi$ for any attainable equilibrium $\mu$ and $R(x,r) = B_{\mu}(r) - px$).
    The objective function is then rewritten as a function of the pdf $f_\mu$ using the expression of the reward.
\end{proof}

We now relax~\eqref{eq:reformulated_problem} by ignoring the decreasingness of the additional reward in~\eqref{eq:reformulated_problem}:
\begin{equation}\label{eq:relaxed_problem}\tag{\mbox{$\widetilde{P}^\text{ret}$}}
\min_{f:\R\to\R_+} \left\{ \left.\kappa\left(\int_{\R}y f(y) dy\right) + 2c\sigma^2\int_{\R} \ln\left(\frac{f(y)}{\fnom(y)}\right)f_\mu(y)dy \;\right\vert  \int_{\R} f(y) dy = 1
\right\}\enspace.
\end{equation}
The discussion about the relation between the initial problem \eqref{eq:reformulated_problem}  and the relaxed one \eqref{eq:relaxed_problem} is provided further. The optimal solution of this  relaxed problem is then characterized by the following lemma:

\begin{lemma}[Characterization of the optimal distribution for the relaxed problem]\label{lemma::optimal_reward}Let ~\Cref{hypo::saving_function} holds. Then, \eqref{eq:relaxed_problem} defines a convex problem. Moreover, if $\mu^*$ admits a density $f_{\mu^*}$ which minimizes~\eqref{eq:relaxed_problem}, then it satisfies the following optimality conditions: for $\mu^*$-almost every $x$ in $\R$,

\begin{align}\label{eq::fixed_point_fmu}
 f_{\mu^*}(x) = \frac{1}{\alpha({\mu^*})} \fnom(x) \exp\left(-x \frac{\kappa'(m_{\mu^*})}{2c\sigma^2}\right)
\end{align} 
 where
$$\alpha(\mu) = \int_{\R} \fnom(y) \exp\left(-y \frac{\kappa'(m_\mu)}{2c\sigma^2}\right)dy\enspace.$$
Conversely, any distribution with density function that satisfies~\eqref{eq::fixed_point_fmu} is optimal for~\eqref{eq:relaxed_problem}.
\end{lemma}
\begin{proof}
The convexity of the objective functional with respect to $f$ comes from the convexity of $\kappa$ (see~\Cref{hypo::saving_function}) and the convexity of $x\mapsto x\ln(x)$. The first-order conditions for~\eqref{eq:relaxed_problem}  are detailed in~\Cref{app::optimal_reward}. Furthermore, they are sufficient for this convex problem; see, e.g.~\cite[Theorem 3.3]{Lanzetti_2022}. 
\end{proof}
In contrast to \cite{Bayraktar_2019}, the optimal distribution is no longer explicit due to the general function $\kappa(\cdot)$. Instead, the optimal distribution is implicitly known through the fixed-point equation~\eqref{eq::fixed_point_fmu}. We simplify this condition in the following theorem to end up with a one-dimensional fixed-point equation on the mean consumption.

\begin{lemma}(Characterization via a fixed-point equation)\label{prop::optimal_reward}
Let~\Cref{hypo::saving_function} holds, and let $\delta:\mathbb{R} \to \mathbb{R}$ be a function given by  $$\delta(m) = p - \kappa'(m)\enspace.$$ Then, if $\mu^*$ admits a density function, which minimizes~\eqref{eq:relaxed_problem}, then it satisfies the following optimality condition: $\mu^*$ is a Gaussian process of mean $m^*$ and standard deviation $\sigma\sqrt{T}$,
where $m^*$ satisfies the fixed-point equation 
\begin{equation}\label{eq::fixed_point_m}
    m - \xpi = \frac{T}{2c}\delta(m)\enspace.
\end{equation}
Conversely, a distribution following a normal distribution $\mathcal{N}(m^*,\sigma\sqrt{T})$ is optimal for~\eqref{eq:relaxed_problem}.
\end{lemma}

\begin{theorem}\label{corol::decreasing}
Let~\Cref{hypo::saving_function} holds. Then, the fixed-point equation~\eqref{eq::fixed_point_m} admits a unique solution $m^*\in]0,\xpi]$. As a consequence, $\eqref{eq:relaxed_problem}$ admits an unique optimal solution.
    Moreover, the associated reward function $B_{\mu^*}$, defined as
\begin{equation}\label{eq::optimal_reward}B_{\mu^*}(r) = \tau \xnom + \frac{c}{T}\left[(\xpi)^2 - (m^*)^2\right] + q_{\mu^*}(r) \delta(m^*) \enspace,\end{equation}
is decreasing.
The associated retailer gain is then
\begin{equation}\pi^* = m^* \kappa'(m^*) - \kappa(m^*) + \left(\frac{m^* + \xpi}{2}\right)\delta(m^*)- \tau \xnom\enspace.\end{equation}
\end{theorem}
\begin{proof}
The increasingness of $\kappa'(\cdot)$ suffices to ensure  that~\eqref{eq::fixed_point_m} admits a unique solution. The existence and uniqueness of a solution for~\eqref{eq::fixed_point_m} entail the existence and uniqueness of an optimal reward for~\eqref{eq:relaxed_problem}; see~\eqref{prop::optimal_reward}. Moreover, we know from~\Cref{hypo::saving_function} that $\delta(0)\geq 0 \geq \delta(\xpi)$. Thus, the root $m^*$ of equation $\tfrac{T}{2c}\delta(m)-m+\xpi=0$ must belong to $]0,\xpi]$. As a consequence, $\tfrac{T}{2c}\delta(m^*) = m^* - \xpi\leq 0$ and the reward function $B_{\mu^*}$ is decreasing.
\end{proof}
The knowledge of the bounds for $m^*$ along with the decreasingness of $\delta(\cdot)$ allows to use, for instance, a binary search algorithm to numerically find the optimal mean consumption in logarithmic time.
\begin{remark}For quadratic function $s:m\mapsto \alpha^s_2 m^2 + \alpha^s_1 m + \alpha^s_0$ and affine cost function $c_p: m\mapsto \alpha^c_1 m + \alpha^c_0$, the fixed point of~\eqref{eq::fixed_point_m} is analytically known:
    $$m^* = \left({1+\tfrac{\alpha^s_2 T}{c}}\right)^{-1}\left(\xnom - \tfrac{(\alpha^s_1 + \alpha^c_1) T}{2c}\right)\enspace.$$
\end{remark}
The function $\delta(\cdot)$ is here interpreted as the \emph{reduction desire} of the provider, since the consumption reduction $\xpi - m^*$ is proportional to $|\delta(m^*)|$, see~\eqref{eq::fixed_point_m}. It expresses the marginal benefit that comes from selling electricity (including the penalty function $s$ provided by the regulator).

In the relaxed problem, we neglect that the reward is decreasing. However, this is directly ensured by~\Cref{corol::decreasing}: the reward provided in~\Cref{prop::optimal_reward} decreases if and only if $\delta$ is negative at the optimum. Therefore, it is also optimal for the original retailer problem~\eqref{eq::master_problem}.

The optimal reward obtained in Eq.~\eqref{eq::optimal_reward} is defined through the quantile of $\mu^*$ and is therefore unbounded.
From the application viewpoint (it is not realistic to give unbounded rewards to consumers) and for numerical issues, we now look at truncated reward. For this purpose, let us define for any $M>0$ the truncated optimal equilibrium distribution $\mu_M$ through its p.d.f:
\begin{equation}\label{eq::fM}
    f_{\mu_M}(x) \propto h_M(x) :=\fnom(x) \exp\left(\frac{-x\kappa'(m_{\mu^*}) \wedge M \vee (-M)}{2c\sigma^2}\right) \enspace.
\end{equation}
In~\eqref{eq::fM}, $f_{\mu_M}$ is equal to $h_M$ up to a multiplicative constant so that $\int_{\R} f_{\mu_M} = 1 $. 
\begin{theorem}[Bounded reward]\label{prop::bounded_reward}
    The total reward which leads to equilibrium $\mu_M$ and gives to agents the utility $\Vpi+\tau \xnom$ is bounded for every consumption level and is defined as 
    \begin{equation}\label{eq::bounded_reward}
    \forall x\in\R,\; R(x,F_{\mu_M}(x)) = \Vpi + \tau \xnom -2c\sigma^2\ln \int_{\R} h_M(y)dy + x \kappa'(m_{\mu^*})\wedge M \vee (-M)\enspace.
    \end{equation}
    Moreover, the mean consumption converges to the optimal one : $$m_{\mu_M} = m_{\mu^*} + O\left(e^{-\frac{M}{2c\sigma^2}}\right)\enspace.
    $$
\end{theorem}
\begin{proof} From~\Cref{prop::map_reward}, the total reward associated to $\mu_M$ is $R_{\mu_M}=\Vpi+\tau \xnom + 2c\sigma^2\ln\left(f_{\mu_M}(y)/\fnom(y)\right)$ and satisfies the utility constraint by construction. The result is then obtained using the definition of $f_{\mu_M}$. 
Besides, one can show (see~\cite[Theorem 5.4]{Bayraktar_2019}) that $\int_{\R} h_M dx = \alpha(\mu^*) + O\left(e^{-\frac{M}{2c\sigma^2}}\right)$ and $\int_{\R} x h_M(x) dx = \alpha(\mu^*) m_{\mu^*} + O\left(e^{-\frac{M}{2c\sigma^2}}\right)$. As a consequence, $m_{\mu_M} = m_{\mu^*} + O\left(e^{-\frac{M}{2c\sigma^2}}\right)$.
\end{proof}
As the optimal (unbounded) total reward, its truncated analogue obtained in~\eqref{eq::bounded_reward} is -- at the equilibrium -- linear in the cumulated consumption over the period (inside the bounds $[-M,M]$).  This means that the consumers are rewarded proportionally to their consumption reduction. Moreover, for both the theoretical bonus~\eqref{eq::optimal_reward} and the bounded one~\eqref{eq::bounded_reward}, $\tau$ only acts as a shift on the function in order to uplift or lower the bonus received by each agent. Consequently, it is possible to a posteriori choose $\tau$ in such a way that the bonus of a given ranking corresponds to a certain amount.

% -------
\subsubsection{Heterogeneous population}\label{subsec::heterogeneous_pop}
We consider here the more general setting of a heterogeneous population, not studied yet in the ranking games literature, which consists in a finite number of clusters $K>1$. The transformation which leads to~\eqref{eq:reformulated_problem} still applies, but the additional constraint in~\Cref{hypo::identical_reward},(ii) has to be imposed to ensure the unitary reward is identical for every sub-population\footnote{
Using~\Cref{prop::map_reward}, there exists a common unitary reward leading to equilibrium $\mu_1,\hdots,\mu_K$ if and only if  there exists 
for all $k\in[K]$ a constant $C_k$ such that $\frac{c_k\sigma_1^2}{\xnom_k}\ln \left(\frac{f_{\mu_k}(x)}{\fnom_k(x)}\right) = \frac{c_1\sigma_1^2}{\xnom_1}\ln \left(\frac{f_{\mu_1}(x)}{\fnom_1(x)}\right) + C_k$ for all $x\in\R$.
}.
%The reformulated optimization problem reads as follows
%\begin{equation}\label{eq:reformulated_problem_heterogeneous}
%\begin{aligned}
%\min_{\mu_1,\hdots,\mu_K} &\quad \kappa\left(\sum_{k\in[K]}\rho_k\int_{-\infty}^{+\infty} y f_{\mu_k}(y) dy\right) + \sum_{k\in[K]}2\rho_k c_k\sigma_k^2\int_{-\infty}^{+\infty} \ln\left(\frac{f_{\mu_k}(y)}{\fnom_k(y)}\right)f_{\mu_k}(y)dy\\
%\text{s.t.} &\quad  \int_{-\infty}^{+\infty} f_{\mu_k}(y) dy = 1,\quad\forall k \in[K]\\
%&\quad y\mapsto\ln\left(\frac{f_{\mu_k}(y)}{\fnom_k(y)}\right) + \frac{p}{2c_k\sigma_k^2}y \text{\quad decreasing}\\
%&\quad \Vpi_1 + 2c_1\sigma_1^2\ln\left(\frac{f_{\mu_1}(y)}{\fnom_1(y)}\right) = \frac{\xnom_1}{\xnom_k}\left[\Vpi_k + 2c_k\sigma_k^2\ln\left(\frac{f_{\mu_k}(y)}{\fnom_k(y)}\right)\right],\; \forall y \in\R, k\in[K]
%\end{aligned}
%\end{equation}
%Even if we relax the problem as in the homogeneous case by removing the constraint of monotonicity, the first-order conditions are not sufficient to ensure optimality, due to the nonlinearity of the last set of equality constraints. In fact, this set of constraints represents the coupling condition of having the same unitary bonus for every sub-population (\Cref{hypo::identical_reward}, item ii).

As it will be seen below, we can recover explicitly solvable problems for a subclass of heterogeneous populations for which all agents of the overall population are similar up to a scaling factor.

\begin{prop}[Explicit characterization for a sub-class of heterogeneous population]\label{prop::optimum_heterogeneous}
Let suppose that the following statement holds:
\begin{equation}\label{eq::condition_scaling}
\forall k\in[K],\quad \frac{\xnom_{k}}{\xnom_{1}} = \frac{\sigma_{k}}{\sigma_{1}} = \frac{c_{1}}{c_{k}} \quad(:=\theta_k)\enspace.
\end{equation}
Then, any $\mu_1,\hdots,\mu_K$ equilibrium distributions associated to a common unitary reward $\beta$ solution of~\eqref{eq::master_problem} satisfies $f_{\mu_k}(y) = \tfrac{1}{\theta_k}f_{\mu_1}\left(\tfrac{y}{\theta_k}\right)$ for all $k\in[K]$. Moreover, the retailer's profit problem simplifies to 
\begin{equation}\label{eq::master_problem_scaling}
\pi^*:=\bar{\theta}\max_{\beta\in \calB}  \left\{p m_{\mu_1} - \tilde{\kappa}(m_{\mu_1})- \xnom_1\int_{0}^{1} \beta(r)dr\; \left| \; \begin{aligned}
    &R_1(x,r) = \xnom_1 \beta(r) - px\\
    &\mu_1 = \epsilon_1(R_1)\\
    &V_1(R_1,\mu_1) \geq \Vpi_1 +\tau \xnom_1
\end{aligned} \right.\right\}\enspace,
\end{equation}
with $\tilde{\kappa}(m) = \bar{\theta}^{-1}\kappa(\bar{\theta} m )$ and $\bar{\theta}=\sum_{k\in[K]}\rho_k\theta_k$.
\end{prop}
\begin{proof}
    Using the characterization of the equilibrium in~\eqref{eq:q_mu}, $q_{\mu_k}(r) = \theta_k q_{\mu_1}(r)$. Therefore, $F_{\mu_k}(y) = F_{\mu_1}\left(\tfrac{y}{\theta_k}\right)$ and $f_{\mu_k}(y) = \tfrac{1}{\theta_k}f_{\mu_1}\left(\tfrac{y}{\theta_k}\right)$.  Moreover, 
    $$
    \begin{aligned}
    \gamma(\mu_k) &= \int_{\R}\fnom_k(x)\exp\left(\frac{\xnom_k\beta(F_{\mu_k}(x))-px}{2c_k\sigma_k^2}\right)dx 
    \\
    &= \int_{\R}\tfrac{1}{\theta_k}\fnom_1\left(\tfrac{x}{\theta_k}\right)\exp\left(\frac{\xnom_1\beta\left(F_{\mu_1}\left(\tfrac{x}{\theta_k}\right)\right)-p\tfrac{x}{\theta_k}}{2c_1\sigma_1^2}\right)dx
    = \gamma(\mu_1)\enspace.
    \end{aligned}$$
    Therefore, $V_k(R_k,\mu_k) = \theta_k V_1(R_1,\mu_1)$. As $\Vpi_k = \theta_k\Vpi_1$, the utility constraint is satisfied for every sub-population.
\end{proof}
\Cref{prop::optimum_heterogeneous} shows that in this specific case of heterogeneous population, the problem boils down to the homogeneous framework, up to a re-scaling of the cost function $\kappa$. Therefore, \Cref{prop::optimal_reward,corol::decreasing,prop::bounded_reward} still apply, and in particular, the optimal distribution is $\mu^*_1=\mathcal{N}(m_1^*,\sigma_1\sqrt{T})$ where $\mu^*_1$ is uniquely determined by the equation $m_1^* -\xpi_1 = \frac{T}{2c_1}(p-\tilde{\kappa}'(m_1^*))$. The condition~\eqref{eq::condition_scaling} corresponds to the case where (i) the volatility of the noise is proportional to the nominal consumption and where (ii) the price elasticity is identical for all sub-populations (see~\Cref{sec::instances} and~\eqref{eq::elasticity} for the link between the cost of effort $c_k$ and the elasticity). 
The second statement (ii) may be more debatable, as the elasticity of a consumer intuitively depends on the equipment of the housing (for instance the type of heating).

% =============================================================================================

% ---------------------------
\section{Numerical resolution in the non-uniform heterogeneous case}\label{sec::reward_optim}
%\subsection{Restriction to piecewise linear reward}
%\textcolor{magenta}{R: Est-ce qu’on peut dire que l’on calcule le reward optimal pour le problème de départ, alors que toutes les convergences et résultat n’ont pas été rigoureusement démontrés? Il serait peut-être préférable d’ajouter une petite précision pour signaler que cette approche reste assez empirique dans cette section.  Meme si l’existence d’un optimum pour le problème initial est garantie,  ce serait peut etre bien de préciser que l’on résout une approximation du problème, en travaillant sur un sous-espace dense?}
\paragraph{Restriction to bounded piecewise linear rewards}
As detailed in section~\ref{subsec::heterogeneous_pop}, solving analytically the non-uniform heterogeneous case is much harder due to the coupling constraint that imposes a common reward function across the sub-populations. Therefore, we develop a numerical algorithm to compute the best \textit{decreasing bounded piecewise linear} reward for the original problem~\eqref{eq::master_problem}, which is a novel approach in the context of mean-field principal-agent ranking games.
 To this end, for a given $N\in\N$, we denote by $\Sigma_N$ the uniform discretization of the interval $[0,1]$ by $N$ points, such that
$\Sigma_N := \{0 = \eta_1 < \eta_2 < \hdots < \eta_N = 1\}$.
Let $M\in\R_+$, then we define the class of bounded piecewise linear rewards adapted to $\Sigma_N$ as $$
\widehat{\mathcal{B}}^N_M
:= \left\{r \in [0,1] \mapsto \sum_{i=1}^{N-1}\mathds{1}_{r \in [\eta_i,\eta_{i+1}[}\left[b_i + \frac{b_{i+1} - b_i}{\eta_{i+1} - \eta_i}(r - \eta_i) \right]\quad\left|\quad\begin{aligned} &b\in[-M,M]^N\\
&b_1 \geq \hdots \geq b_N
\end{aligned}\right.\right\}\enspace.
$$ 
The reward function obtained  as a linear interpolation of a non-increasing vector $b$ is denoted by $\hat{\beta}[b]$.
For this special class of rewards, the computation of some integrals can be simplified. The integral that appears in the equilibrium characterization~\eqref{eq:q_nu} becomes
$$
\begin{aligned}
&\int_0^1 \exp\left(-\frac{\xnom_k\hat{\beta}[b](r)}{2c_k\sigma_k^2}\right) dr \\
%= 2c_k\sigma_k^2 (\xnom_k)^{-1} \sum_{i=1}^{N-1}\frac{\eta_{i+1} - \eta_i}{b_{i+1} - b_i} \left[\exp\left(-\frac{\xnom_k b_{i+1}}{2c_k\sigma_k^2}\right) - \exp\left(-\frac{\xnom_k b_{i}}{2c_k\sigma_k^2}\right)\right]$$
&=\sum_{i=1}^{N-1}\begin{cases}
    \frac{2c_k\sigma_k^2}{\xnom_k} \frac{\eta_{i+1} - \eta_i}{b_{i+1} - b_i} \left[\exp\left(\frac{-\xnom_k b_{i+1}}{2c_k\sigma_k^2}\right) - \exp\left(\frac{-\xnom_k b_{i}}{2c_k\sigma_k^2}\right)\right] \text{ if } b_{i+1} < b_i,\\
    (\eta_{i+1} - \eta_i) \exp\left(-\frac{\xnom_k b_{i}}{2c_k\sigma_k^2}\right) \text{ if } b_{i+1} = b_i
\end{cases}
\end{aligned}
$$
This integral is a continuous function of $b$, as the value for $b_{i+1} = b_i$ corresponds to the limit of the expression for $b_{i+1}<b_i$.
Also, the integral of the bonus simplifies into
$$\int_0^1 \hat{\beta}[b](r) dr = \sum_{i=1}^{N-1}(\eta_{i+1} - \eta_i)\left(\frac{b_{i+1} + b_i}{2}\right)\enspace.$$

\begin{remark}
    Even if we restrict in this numerical section the study to bounded piecewise linear rewards, the problem remains a mean-field game, i.e., the population's distribution is not discretized. In particular, the expression of the agents' best-response~\eqref{eq:q_mu} still applies.
\end{remark}

\paragraph{Box maximization.} 
We denote by $\pi_\lambda:\calB \to \R$ the Lagrangian function of~\eqref{eq::master_problem}, defined as
\begin{equation}\label{eq::lagrangian_function}
\pi_\lambda(\beta) :=  
\left\{\left.
\begin{split}
p m_{\mu} - \kappa(m_\mu) - \sum_{k\in[K]}\rho_k\xnom_k\int_{0}^{1} \beta(r)dr \\
- \lambda \sum_{k\in[K]}\rho_k \left(\Vpi_k+\tau \xnom_k - V_k(R_k,\mu_k)\right)^+
\end{split}
\;\right\vert\; 
\begin{split}
    &R_k(x,r) = \xnom_k \beta(r) - p x\\
    &\mu_k = \epsilon_k(R_k)
\end{split}
\right\}\enspace,
\end{equation}
where $(\cdot)^+ := \max(0,\cdot)$. For a given parameter $\lambda > 0$, $\pi_\lambda$ constitutes a relaxed version of the initial problem~\eqref{eq::master_problem}, where violations of the utility condition are not fully forbidden but rather strongly penalized in the objective for large values of $\lambda$. We then focus on the maximization of the Lagrangian function $\pi_\lambda$ over the class of bounded piecewise linear rewards, i.e.,
$\sup_{\beta \in \widehat{\calB}^N_M} \pi_\lambda(\beta) \enspace.$ 
%\todo{R: if you talk about max, the justification of its existence has been put before}
To this end, we define the following transformation:
\begin{equation}
\begin{aligned}
\phi^N_M:&\;[\shortminus 1,1]^N &\to&\; [-M,M]^N\\
&\qquad z&\mapsto&\quad b
\end{aligned}
\text{ \;where } \begin{cases}
b_1 = M z_1\\
b_i = \frac{1}{2}(b_{i-1} - M) + \frac{1}{2}(b_{i-1} + M) z_i,\;i>1
\end{cases}\enspace.
\end{equation} %\todo{R: 1.If $z=1$, the inequalities are not strict, but you need strict decreasigness between $b_1$, $b_2..$. \\ 2. For the inverse function, there is a problem if there is no strict decresigness ge
We also define:
$$
\left(\phi^N_M\right)^{\shortminus 1}(b)
:=
\left\{
\begin{aligned}
&z_1 = \frac{1}{M} b_1\\
&z_i = \frac{2 b_i - b_{i-1} + M}{b_{i-1} + M},\;i>1 \text{ and } b_{i-1} > -M\\
&z_i = -1,\; i>1 \text{ and } b_{i-1} = -M
\end{aligned}
\right.$$
Then, for any $b\in\{b\in[-M,M]^N\;|\;b_1\geq \hdots \geq b_N\}$, $z=\left(\phi^N_M\right)^{\shortminus 1}(b)\in[-1,1]^M$ and $b = \phi_M^N(z)$.
%Moreover, $\tilde{b}_i \leq \frac{1}{2}(\tilde{b}_{i-1} - M) + \frac{1}{2}(\tilde{b}_{i-1} + M) = \tilde{b}_{i-1}$.
As an example, \Cref{fig::ex_transfo} displays $(\eta_i,z_i)_{i\in[N]}$ and the corresponding bonus function $\hat{\beta}[\phi_M^N(z)]$. Note that for any vector $b$ such that $b_i>-M$ for all $i$, $z=\left(\phi^N_M\right)^{\shortminus 1}(b)$  is the unique counterimage of $b$ by $\phi_M^N$. However, for vectors $b$ containing values $-M$, other vectors $z$ exist.

\begin{prop}[\textit{Maximization with box constraints}]\label{prop::box_max}
    \begin{equation}
 \sup_{z\in[-1,1]^N} \pi_\lambda(\widehat{\beta}[\phi^N_M(z)]) = \sup_{\beta \in \widehat{\calB}^N_M} \pi_\lambda(\beta) \enspace.
    \end{equation}
\end{prop} %\todo{R: 1. Why the max for the function $\pi_\lambda(\beta)$ over $\widehat{\calB}^N_M$ is attained? It is not clear that the set $\widehat{\calB}^N_M$ is compact, since the ineuqlaities between the points $b_1, b_2,...$ are strict, the set is not closed.\\ 2: For the continuity of the objective function, there is the penalty which included the value function $V$, here to get continuity, one needs strong norm, as you have the $sup$ over the controls}
%\todo{R: if you talk about max, the justification of its existence has been put before}
\begin{proof}
    By definition of $\widehat{\calB}^N_M$, $\sup_{z\in[-1,1]^N} \pi_\lambda(\widehat{\beta}[\phi^N_M(z)]) \leq \sup_{\beta \in \widehat{\calB}^N_M} \pi_\lambda(\beta)$, as $\widehat{\beta}[\phi^N_M(z)]\in \widehat{\calB}^N_M$. Conversely, to each vector $\beta\in\widehat{\calB}^N_M$, there is at least one vector $z\in[-1,1]^M$ such that $\beta = \widehat{\beta}[\phi^N_M(z)]$, see the definition of $\left(\phi^N_M\right)^{\shortminus 1}$, therefore $\sup_{z\in[-1,1]^N} \pi_\lambda(\widehat{\beta}[\phi^N_M(z)]) \geq \sup_{\beta \in \widehat{\calB}^N_M} \pi_\lambda(\beta)$.
\end{proof}

\begin{figure}[!ht]
\centering
\includegraphics[width = 0.6\linewidth]{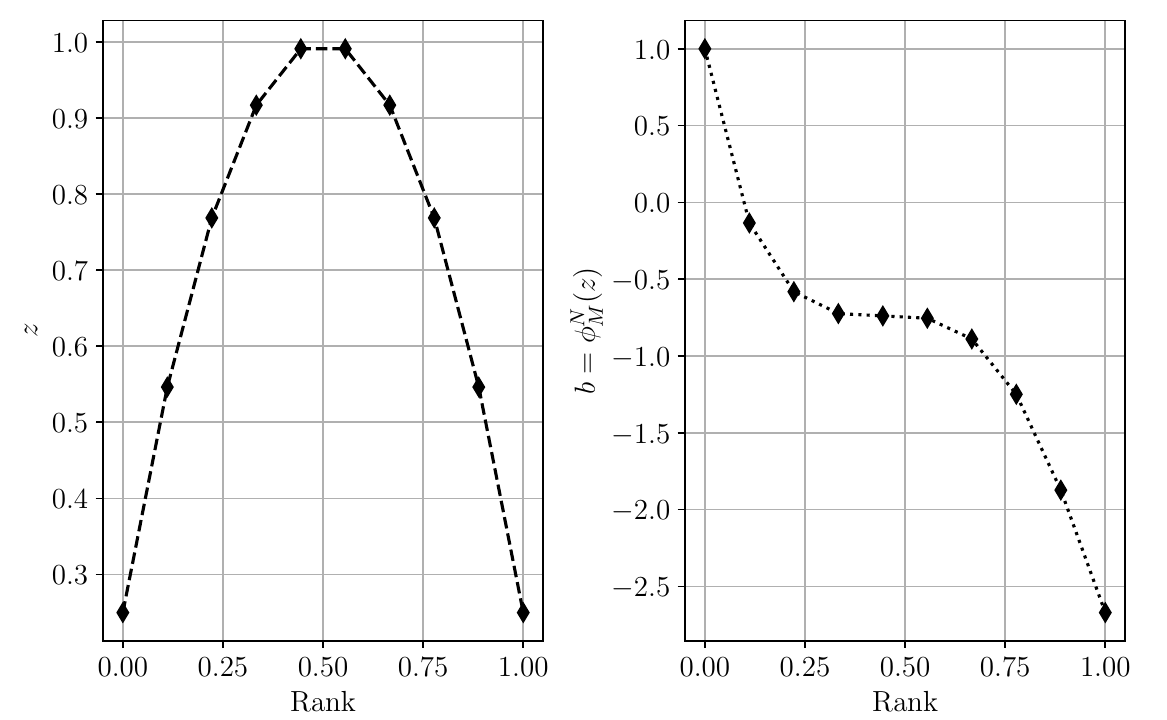}
\caption{Example of transformation using function $\phi^N_M$ for $M = 4$ and $N=10$}
\label{fig::ex_transfo}
\end{figure}
\Cref{prop::box_max} shows that the problem of maximizing the Lagrangian function $\pi_\lambda$ over the class of rewards $\widehat{\calB}^N_M$ is equivalent to the maximization of the continuous function $z\mapsto\pi_\lambda(\widehat{\beta}[\phi^N_M(z)])$ over the box $[-1,1]^N$, for which we show below the existence of an optimal solution (however, contrary to the the homogeneous case in~\Cref{corol::decreasing}, the uniqueness of the optimum is not proven).

\begin{corollary}[Existence of a solution]
    Let $N\in\mathbb{N}$ and $M\in\mathbb{R}_+$. Then, there exists $z^*\in[-1,1]^N$ such that $\pi_\lambda(\widehat{\beta}[\phi^N_M(z^*)]) = \sup_{z\in[-1,1]^N} \pi_\lambda(\widehat{\beta}[\phi^N_M(z)])$, 
    and consequently by defining $\beta^* := \widehat{\beta}[\phi_M^N(z)]$, it holds that $\pi_\lambda(\beta^*) = \sup_{\beta \in \widehat{\calB}^N_M} \pi_\lambda(\beta)$.
\end{corollary}
\begin{proof}
     The functions $b\mapsto\hat{\beta}[b]$ and $z\mapsto\phi^N_M(z)$ are continuous by definition. Also, the mean consumption $m_\mu$ (resp.~the value function $V_k$) that appears in the definition of the Lagrangian function $\pi_\lambda$~\eqref{eq::lagrangian_function} is a continuous function of the reward, see~\eqref{eq:q_nu} (resp.~\Cref{prop::best_response}).
     Therefore, the objective function $z\mapsto \pi_\lambda(\widehat{\beta}[\phi_M^N(z)])$ is continuous and has an optimum over the compact set $[-1,1]^N$.
\end{proof}

From a computational viewpoint, the search space is now independent of $M$, and the decreasingness of the bonus function is directly encoded in the transformation $\phi^N_M$. The only remaining constraints are the ones ensuring that the solution belongs to the unit box. The search is then achieved by black-box optimization, since the evaluation of $\pi_\lambda$ can be explicitly done using~\eqref{eq:q_nu}-\eqref{eq:q_mu}). In the numerical results, we use CMA-ES~(\cite{Hansen_2006}) as optimization solver through the C++ interface (\cite{Fabisch_2013}).
Convergence properties of the solver are analyzed in~\cite{Hansen_1997}, and we display in~\Cref{sec::num_res} the numerical convergence of the objective along the iterations. 
\begin{remark} (i) The evaluation of $\pi_\lambda$ linearly depends on the number of sub-populations (i.e., $K$) since, given a reward, the problem boils down to the computation of the equilibrium distributions for the $K$ sub-populations.\\
        (ii) The reward function found by~\Cref{algo::CMA} is bounded and decreasing, but might violate the utility constraint ``$V_k(R,\mu_k) \geq \Vpi_k+\tau \xnom_k$" for small penalization values of $\lambda$. %Note that, for the case $K=1$, if the optimizer for the discrete problem on a sufficiently precise grid is a global optimizer, then we get an $\varepsilon$-solution of the initial problem, see~\Cref{prop::bounded_reward}.
\end{remark}

\begin{algorithm}[!ht]
\caption{\textsc{Optimization of the reward}}\label{algo::CMA}
\begin{algorithmic}
\Require $M$, $N$, $\lambda$, $\Sigma_N$, solver $\Pi$, initial point $z^0$,
\State Construct $\Theta$ as
\begin{equation}
\Theta : z\in[\shortminus 1,1]^N \mapsto \pi_\lambda \left(\hat{\beta}[\phi_M^N(z)]\right)\end{equation}
\State Apply $\Pi$ to maximize $\Theta$ (starting from $z^0$) and get the final state $z^\Pi$.
\State\Return $\beta^\Pi = \hat{\beta}[\phi^N_M(z^\Pi)]$. 
\end{algorithmic}
\end{algorithm}

% ================================
\section{Application to Energy Savings}\label{sec::num_res}

In this section, we develop a case study related to the French market of Energy Saving Certificates based on the use of realistic data. We compare the results with existing reward mechanisms, and analyze them in terms of consumption reduction (relatively to the target imposed by the European commission).
\subsection{Instances}\label{sec::instances}
\paragraph{Consumers} We consider the case where the retailer aims at designing a reward for 4 types of consumers, listed in~\Cref{table::conso_usage}. Data on the average annual consumption correspond to the French case.
\begin{table}[!ht]
\footnotesize
\centering
  \begin{tabular}{rccccp{3cm}<{\centering}}
    \toprule
     & Distribution & Housing & Heating & Nb occupants & Consumption (mean/year)\\\midrule
Sub-pop. $1$ & 26\% & House 70 m$^{2}$ &  Electric &  3 & 9.9 MWh \\
Sub-pop. $2$ & 49\% & House 70 m$^{2}$ & Non-electric & 3 & 1.5 MWh \\
Sub-pop. $3$ & 9\% & House 150 m$^{2}$ &  Electric &  4 & 20 MWh\\
Sub-pop. $4$ & 16\% & House 150 m$^{2}$ & Non-electric & 4 &  2.2 MWh\\
\bottomrule
\end{tabular}
\caption{Annual electricity consumption by type of usage.\\The consumption data are extracted from ``Agence France Electricité"\protect\footnotemark.}
\label{table::conso_usage}
\end{table}
The consumers are here distinguished according to the surface of the housing and the type of heating, which can represent up to 90\% of the annual consumption. A more elaborated clustering might also take into account the location of the housing or the age of the occupants, but we focus here on the two main factors affecting the consumption. We suppose for simplicity that the overall population is composed of these four sub-populations, representing a total of 33 millions of households (current number of households in France). The distribution of the sub-populations is then computed by considering that there are thrice as many 70m$^2$-houses as 150m$^2$-houses (the mean surface in France\footnote{\url{https://www.lamaisonsaintgobain.fr/blog/insolites/metre-carre-et-confort-connaissez-vous-la-moyenne-francaise}} in around 90m$^2$) and that a 35\%\footnote{\url{https://www.voltalis.com/comprendre-electricite/les-types-de-chauffage-preferes-des-foyers-francais-1772}} of the French households is equipped with electric heating. This gives us a mean annual consumption of $5.46$MWh, or a total annual consumption of 180TWh. In comparison, the French annual consumption for residential households is around 155TWh. This slight over-estimation is due to the fact that we only consider here houses with three or four occupants.

We suppose that the consumption levels displayed in~\Cref{table::conso_usage}\footnotetext{\url{https://www.agence-france-electricite.fr/consommation-electrique/moyenne-par-jour/}} corresponds to customers having subscribed to a regulated offer, corresponding to a fixed price of electricity $p$. As showed in~\Cref{corol::pi}, nominal consumption ($\xnom$) and consumption under price $p$ ($\xpi$) are linked by the relation $\xpi = \xnom - \frac{p}{2c}$ (we consider annual consumption in~\Cref{table::conso_usage}).

In~\cite{Niromandfam_2020}, the authors used several concave utility function to model the price elasticity of the electricity demand. In particular, they studied a quadratic utility function similar to the cost of effort we consider: for a one-year horizon and constant effort, 
$
\Vpi_k = \max_{x\in \R} \{-px - c(x-\xnom)^2\}\enspace.
$
This corresponds to the welfare maximization with quadratic utility, defined as $U(x,\xnom) = -c(x - \xnom)^2$. In other words, we reinterpret -- in a simple case -- the effort cost of the consumer model~\eqref{eq::consumer_problem} as a quadratic utility function in the sense of~\cite{Niromandfam_2020}.
For this type of utility function, the elasticity is defined as $\eta=1-\tfrac{\xnom}{\xpi}$, see e.g.~\cite[Eq. 19]{Niromandfam_2020}).
As a consequence, using the relation between $\xpi$ and $\xnom$ and the definition of the elasticity, one can obtain the following relations:
\begin{equation}\label{eq::elasticity}
    c = \frac{-p}{2\eta\xpi}\enspace,\qquad \xnom = \xpi(1-\eta)\enspace.
    \end{equation}
Several values of price elasticity are reported in~\cite{Niromandfam_2020,Csereklyei_2020}, and we use here $\eta=-0.32$, which corresponds to the estimation of the long-run residential price elasticity made by~\cite{Bonte_2015} on the EPEX spot market between 2012 and 2014. Price elasticity is always studied at the scale of a country (or even broader), and therefore we take an estimate which is identical for all the agents (\emph{uniform} elasticity). In the numerical results, we will analyze the influence of a non-uniform elasticity, see~\Cref{sec::num_res}. 

Regarding the volatility, in the Low Carbon London pricing study, Carmichael et al.~\cite{Carmichael_2014} reported a deviation of $\pm$200 Watt for a demand of 1000 Watt. We take here a deviation $\sigma\sqrt{T}$ equals to 10\% of the total consumption $X_T$ under zero effort for each of the four sub-populations. Finally, we consider here for $p$ the price of the regulated offer (``Tarif Bleu") in 2019, that is 145 \euro/MWh\footnote{\url{https://prix-elec.com/tarifs/evolution/2019}}.

\begin{table}[!ht]
\centering
\footnotesize
\begin{tabular}{rp{3cm}<{\centering}p{2.5cm}<{\centering}}
    \toprule
     & $c_k$ (\euro/MWh$^2$) & $\sigma_k$ (MWh)\\
    \midrule
    Sub-pop. 1 & 24 & 0.57\\
    Sub-pop. 2 & 156 & 0.09\\
    Sub-pop. 3 & 12 & 4.15\\
    Sub-pop. 4 & 107 & 0.13\\
\bottomrule
\end{tabular}
\caption{Cost of effort and volatility parameters.}
\label{table::costs_effort}
\end{table}

\paragraph{Retailer cost} We consider here the year 2019 (just before the energy crisis). In this numerical study, we suppose that the retailer can access to the energy at the production cost. This represents either that the retailer can produce the electricity himself or that there is a perfect competition on the market. This approximation simplifies the discussion and make easier the interpretation of the model but other cost function could be studied. We display in~\Cref{table::costs_retailer} the marginal cost and the annual production for each type of power plants. This representation is a simplified view of the spot power market where the equilibrium is hourly  but this simplification is not limiting in our view the application and any other cost could be taken.

\begin{table}[!ht]
\centering
\footnotesize
\begin{tabular}{rp{3cm}<{\centering}p{2.5cm}<{\centering}}
    \toprule
    Power plant & Marginal cost (\euro/MWh) & Production (TWh)\\
    \midrule
    Hydro/Wind/Solar & 0 to 15 & 115\\
    Nuclear & 30 & 380\\
    Gas & 70 & 30\\
    Coal & 86 & 7\\
    Fuel & 162 & 5\\
\bottomrule
\end{tabular}
\caption{Marginal price and annual production. Source: \emph{RTE Bilan électrique 2019} and \emph{Ademe}}
\label{table::costs_retailer}
\end{table}

By aggregating the production capacities by increasing cost (as in merit order curves for day-ahead markets), we can obtain an estimate of the supply cost according the production, see~\Cref{fig::costs_retailer}. The total cost is then obtained by dividing the supply cost by $0.35$ as this approximately corresponds to the weight of supply in the total cost\footnote{\url{https://www.ecologie.gouv.fr/commercialisation-lelectricite}}. To fit with our situation where we only look at the residential part of the consumption, we shift the cost curve so that a residential consumption of $180$TWh is ``cleared" by a gas power plant (as it is often the case in the day-ahead market) and we regularize it to be differentiable.
We use in~\eqref{eq::master_problem} a mean cost function that depends on the mean consumption of the overall population for notation convenience. Therefore we design the mean cost function $\kappa$ by normalize the $x$-axis of the curve by the 33 millions of households of the overall population.

\begin{figure}[!ht]
    \centering
    \includegraphics[width=0.7\linewidth]{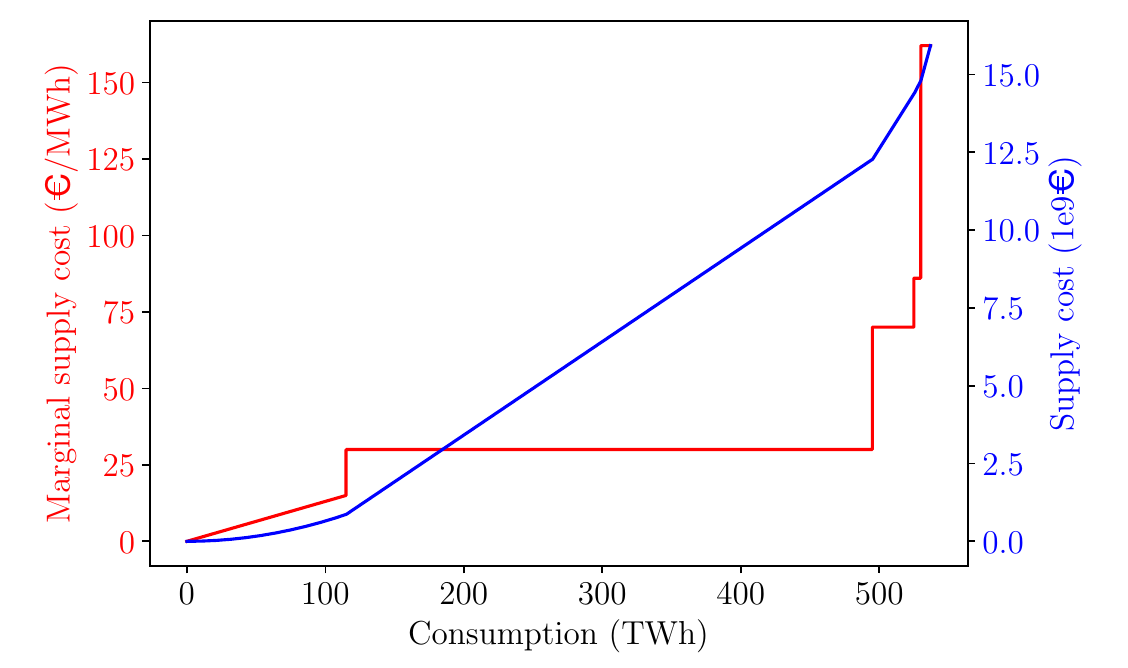}
    \caption{Estimation of supply cost through marginal costs}
    \label{fig::costs_retailer}
\end{figure}

\paragraph{Valuation of energy savings}
Electricity retailers are obliged by the French regulation\footnote{Loi POPE, 2005 : \url{https://www.ecologie.gouv.fr/dispositif-des-certificats-deconomies-denergie}} to reduce the global consumption of their customers, in the context of energy efficiency and sobriety. From 2024 to 2030, the European regulation will impose a reduction target of 1.49\% of the annual consumption, and aspire to reach 1.9\% by the end of 2030. If a retailer does not succeed in gathering a sufficient amount of Energy Saving Certificates, a penalty of 15\euro /MWh is applied (for ``classic" certificates)\footnote{\url{https://www.calculcee.fr/les-primes-cee.php}}. In addition, each provider can buy (resp. sell) on a market a certain quantity of certificates if the quantity of energy consumption overshoots (resp. undershoots) the target. In 2023, the price of certificates is around 7.5\euro /MWh\footnote{\url{https://c2emarket.com/}}. We consider here a target of 5\% of consumption reduction over 3 years ($T=3$), corresponding to a mean consumption of $15.6$MWh for the three years. The valuation function is then defined as $s_\theta(m) = \softplus_\theta(15(m-15.6))$, where $\softplus_\theta = \theta^{\shortminus 1} \log(1+\exp(\theta x))$. \Cref{fig::penalty} shows the two extreme cases : a purely liquid market ($\theta=0$) and the absence of exchange ($\theta=\infty$). We choose here $\theta=0.3$ to represent an intermediate case.
\begin{figure}[!ht]
    \centering
    \includegraphics[width=0.55\linewidth]{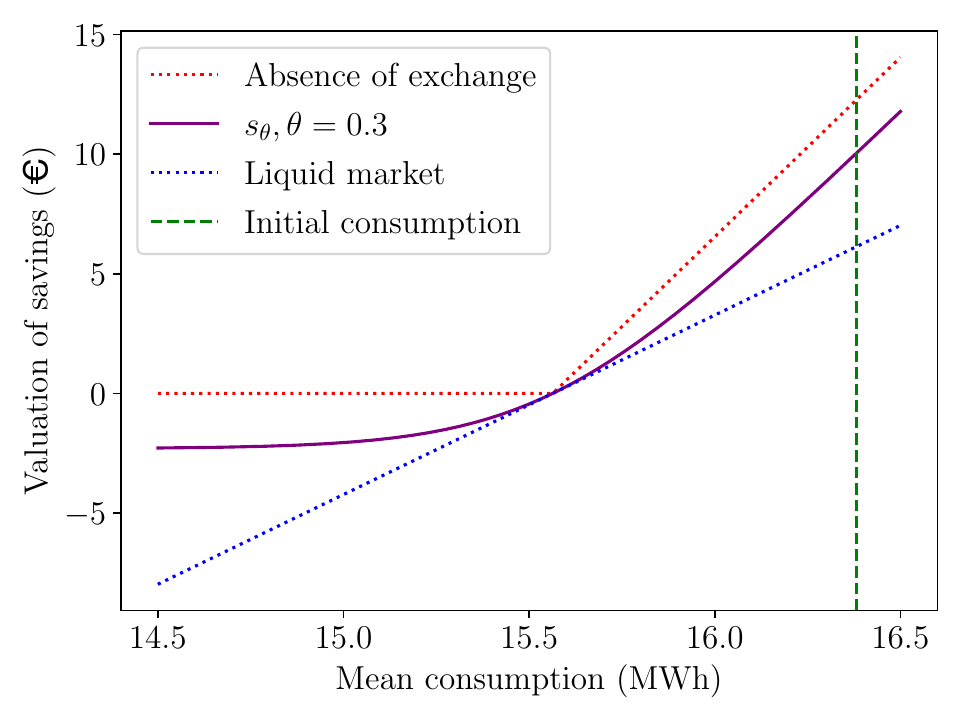}
    \caption{Penalty function $s(\cdot)$ given by the regulator.}
    \label{fig::penalty}
\end{figure}

% -----
\subsection{Numerical Results}

We use $N=20$ discretization points for the bonus description and $M = 0.1p$. This means that the maximal unitary bonus given to an agent cannot exceed $10\%$ of the electricity price. We take $z_0 \equiv 1$ as initial guess. The main advantage of this initial guess is that it satisfies the reservation utility constraint (if $\tau < M$). The step-size parameter of CMA is fixed to 0.05.
The numerical results\footnote{The whole code is available on the GitHub repository: \url{https://github.com/jacquq/rk_games_electricity}.} -- parallelized on 10 threads -- were obtained on a laptop i7-1065G7 CPU@1.30GHz.

\paragraph{Uniform elasticity} \Cref{fig::scalable_case} shows the results for the test case described in~\Cref{sec::instances}, where the price elasticity is identical for all the sub-populations. As a consequence,  \Cref{prop::optimum_heterogeneous} applies and we can analyze in this setting the performance of the numerical solving procedure: in~\Cref{fig::retailer_4k_1}, the reward found by~\Cref{algo::CMA} is very close to the (theoretical) optimal reward, showing that the solver successfully finds the global optimum. 
About the computational cost, the algorithm converged in approximately 3000 iterations (around 400 seconds), but succeeded in reducing the optimality gap to less than 0.5\% in 100 iterations.  

\begin{figure}[!ht]
    \centering
    \begin{subfigure}{.47\textwidth}
        \centering
        \includegraphics[width=\linewidth,clip=true,trim=.3cm 0cm 1.3cm 1.4cm]{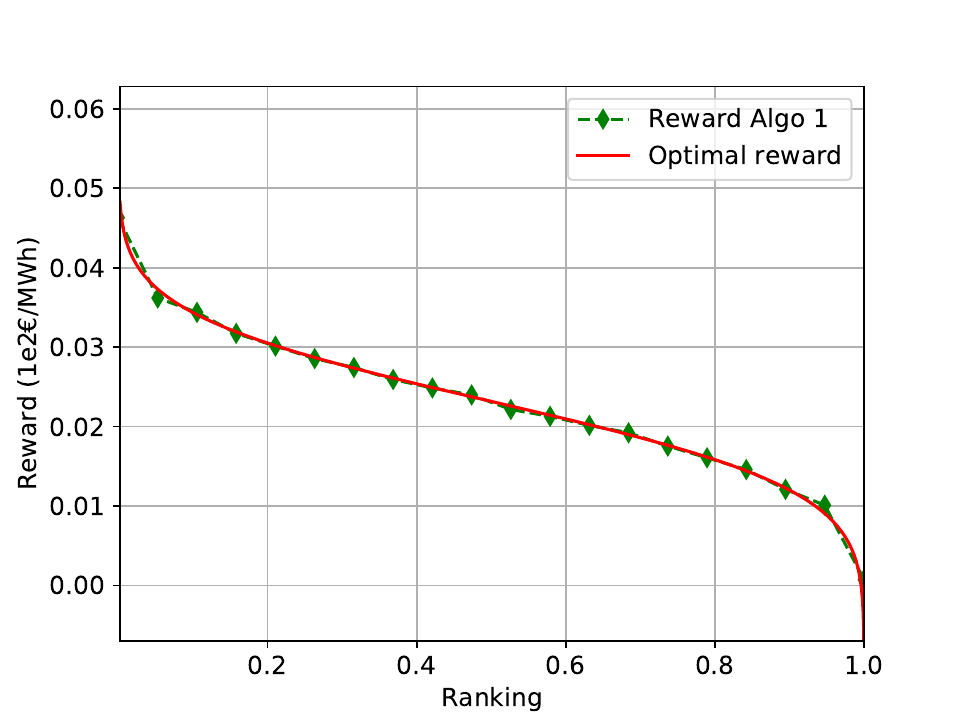}  
        \caption{Analytic optimal reward in red, compared to the unitary bonus function found by~\Cref{algo::CMA}.}
        \label{fig::retailer_4k_1}
    \end{subfigure}
    \hspace{0.2cm}
    \begin{subfigure}{.47\textwidth}
        \centering
        \includegraphics[width=\linewidth,clip=true,trim=.3cm 0cm 1.3cm 1.4cm]{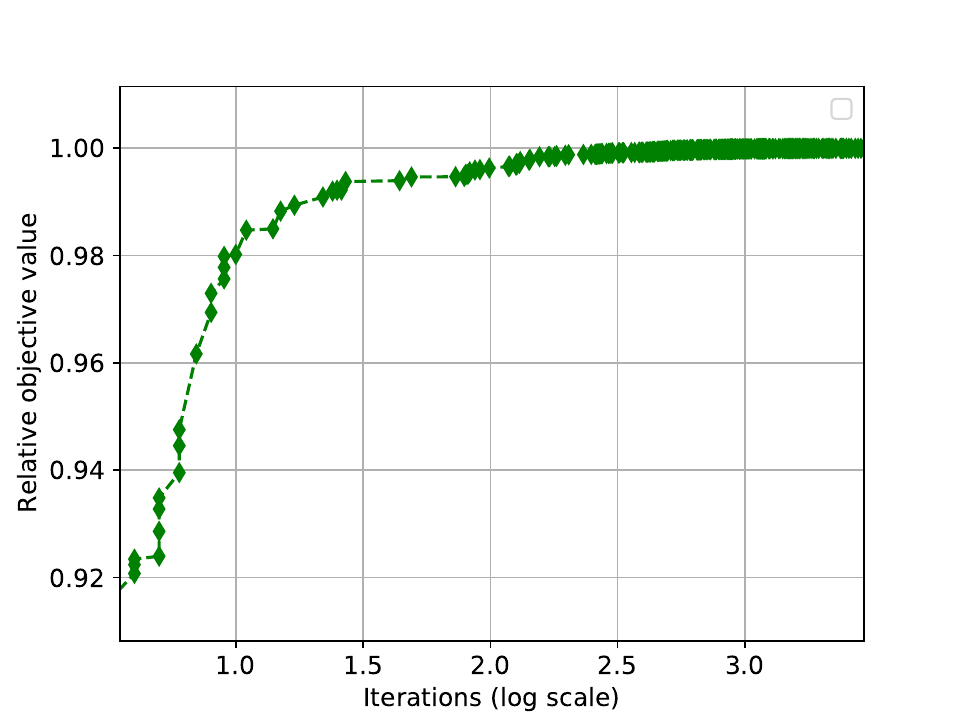}  
        \caption{Evolution of the relative objective value along the iterations.}
        \label{fig::convergence_4k_1}
    \end{subfigure}
    \begin{subfigure}{.98\textwidth}
        \centering
        \includegraphics[width=\linewidth,clip=true,trim=0cm 0cm 0cm 1cm]{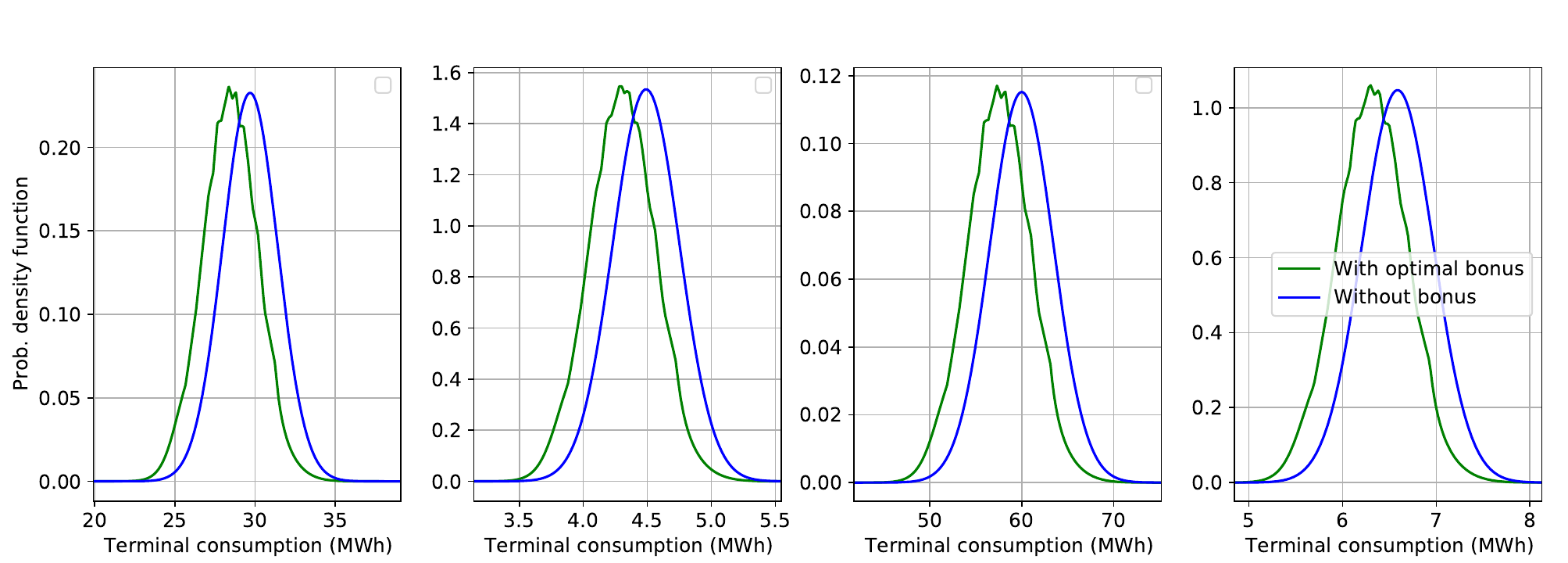}  
        \caption{Distribution of the cumulated consumption over the period $[0,T]$ for the four sub-populations}
        \label{fig::pdf_4k_1}
    \end{subfigure}
    \caption{Numerical results for the four populations described in~\Cref{table::conso_usage,table::costs_effort} (scalable case).}
    \label{fig::scalable_case}
\end{figure}

We depict in~\Cref{fig::scalable_case} the distribution of the cumulated consumption over the period $[0,T]$ for the four sub-populations with and without the bonus. As shown in~\Cref{corol::pi}, the distribution without reward is a Gaussian process centered in $\xpi$ (which corresponds to three times the annual consumption displayed in~\Cref{table::conso_usage}). The terminal distribution with the optimal reward is then a shift of this normal distribution -- see~\Cref{prop::optimum_heterogeneous}. We observe that, as expected, the terminal distribution is also identical for the four sub-populations, up to a scaling ($f_{\mu^*_k}(x) = \theta_k^{\shortminus 1}f_{\mu^*_1}(\theta_k^{\shortminus 1}x)$). Here, the mean pluriannual consumption in the whole population decreased from 16.38MWh to 15.7MWh, giving a saving ratio of 4.1\%. This must be compared with the initial objective of the regulator (a reduction of 5\% of the pluriannual consumption): the retailer found a compromise between the penalty imposed by the regulator, the cost to propose a reward mechanism, and its natural willingness to sell electricity. Note that~\Cref{fig::scalable_case}, (c) shows that this reduction in consumption does not increase variance.

The optimal bonus offered to customers takes a very intuitive shape. 
Indeed, it is rather linear with the rank except for the extreme ranks for which the incentive is amplified. We expect this form of contract to be easy to describe to customers and this facilitates the implementation in practice. We observe negative values for the 1\% consuming the most  (we choose $\tau$ a posteriori in this sense) and goes up to more than 4\euro ~per MWh, which corresponds to a bonus of 66\euro ~in average over the three years. This should be compared for instance with the ``Bonus Conso" proposed by TotalEnergies\footnote{\url{https://totalenergies.com/fr/actualites/communiques-de-presse/bonus-conso-hiver-2023-2024-totalenergies-recompense-ses-clients}}, where 30\euro ~are proposed for a reduction of 5\% over one year. This confirms that the contract we designed is realistic and can be accepted by customers in practice.

%shows the reward found by~\Cref{algo::CMA}. As a comparison, the optimal reward (computed with~\eqref{eq::optimal_reward}) is also drawn. The two reward are very close, meaning that the algorithm has converged to the global optimum. \Cref{fig::cdf_1k} depicts three cumulative distribution function. The nominal one has a mean value of $18$, the one with the price as unique incentive has a mean value around $17$, and the cdf obtained with the optimal reward has a mean value around $15$. As expected, the additional reward has induced a higher effort in the population, and so a higher energy reduction.Note that the consumers has an average bill of $p\xpi = 2\,890$\euro\; (for $T=3$ years) and the additional reward takes values in $[-500,1\;000]$\euro\,. Therefore, the additional reward is, in this toy model, important in comparison to the original bill.

\paragraph{The $N$-players game}
We now numerically illustrate the behavior of several individual consumers incentivized by the optimal bonus found in~\Cref{fig::retailer_4k_1}. The simulation of the trajectories is done using a Euler-Maruyama scheme, see e.g.~\cite{Ngo_2015} for details on the discretization, as for convergence rates.  

\begin{figure}[!ht]
    \centering
    \includegraphics[width=0.55\linewidth,clip=true,trim=.2cm 0cm 1.3cm 1.2cm]{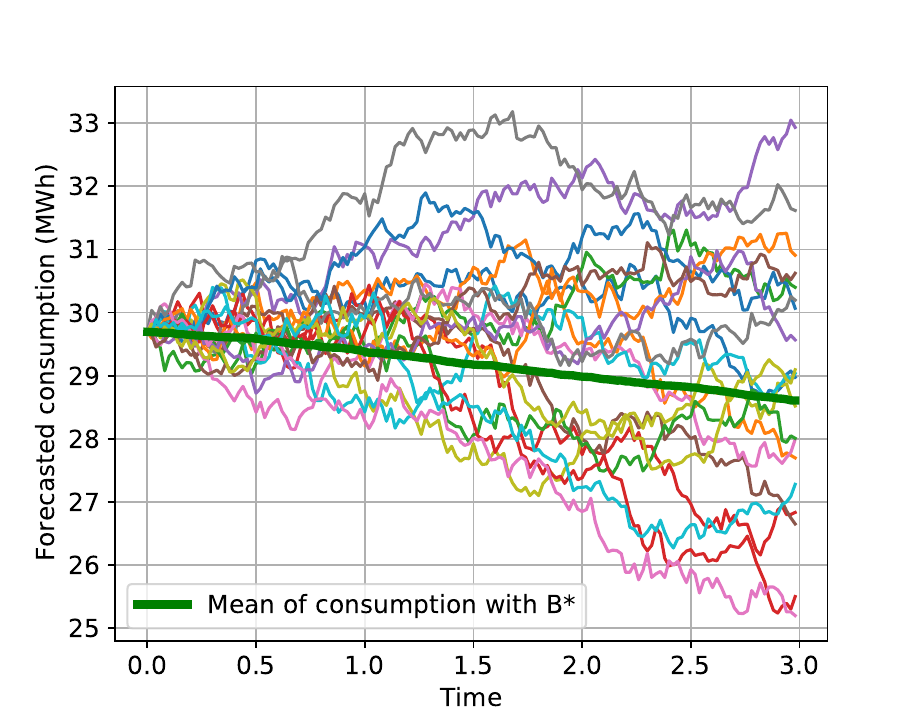}  
  \caption{Deviation of the consumption from the no-bonus case\\~
  Trajectories for 20 consumers from sub-population $1$.}
  \label{fig::traj_4k_1}
\end{figure}

\Cref{fig::traj_4k_1} displays the evolution of the forecasted consumption $X^{a_1^*}_1$, from which we subtracted the deviation coming from price in order to clearly distinguish the supplementary effort made through the influence of the bonus. This corresponds to the quantity 
$$
Y_1(t) = X_1^{a^*_1}(t) + \frac{p(t-T)}{2c_1}\enspace,
$$
where $a^*_1$ is the optimal effort in the presence of the bonus. 
We observe the same consumption decrease as in~\Cref{fig::pdf_4k_1}, and this reduction has a linear behavior. Indeed, we showed in~\eqref{eq::bounded_reward} that the optimal total reward is linear in $x$, and for any reward $R_{k,\mu} = \alpha_0 - \alpha_1 x$, the corresponding effort is $a_k^*(t) = -\frac{\alpha_1}{2c_k}$ -- see~\eqref{eq::optimal_effort} -- and the consumption reduction is then $\tfrac{\alpha_1}{2c_k} t$. This has a strong implication on the behavior of the model: the effort made at time $0\leq t \leq T$ by a consumer is independent from his current situation, i.e., is not influenced by the hazard $W_t$. This means that a consumer will not stop/reduce his effort even if he is undergoing an adverse hazard.

%------
\paragraph{Non-uniform price elasticity} We now slightly change the previous test case by considering that the price elasticity is not constant across the population, but rather depends on the characteristics of each agent. In particular, we consider here that the price elasticity of a consumer with electric heating is greater than someone with another heating technology. This greater specific adaptability is for instance exploited by some energy providers\footnote{\url{https://www.sowee.fr/}}.
To see the influence of non uniform elasticity, we divide by two the elasticity of sub-populations 2 and 4 -- as they do not have electric heating -- and multiply by 1.5 the elasticity of sub-populations 1 and 3. In this setting, the scaling condition~\eqref{eq::condition_scaling} is no longer satisfied, and so, contrary to the previous case, we are not able to find the theoretical optimal bonus function, but only able to perform a numerical optimization using~\Cref{algo::CMA}.

\Cref{fig::nonscalable} shows the results for the test case with modified elasticity parameters. We use here $N=40$ discretization points and let the algorithm runs up to 5000 iterations. The convergence of~\Cref{algo::CMA} is still fast since the gap between the solution at iteration 100 was already close to the final solution to less than 1\%.  About the cumulated consumption distribution, we observe that the mean consumption for sub-populations 1 and 3 is reduced by 5.3\% whereas the mean consumption for sub-populations 2 and 4 is reduced by 2.3\%. Indeed, it reflects the increase (resp. decrease) of price-elasticity for 1 and 3 (resp. 2 and 4). This should be compared with the uniform consumption reduction of 4.1\% in the previous setting.  

The unitary bonus found by~\Cref{algo::CMA} is lower than in~\Cref{fig::scalable_case}: for example, in the uniform-elasticity case, every agent with a ranking lower than 0.6 received a unitary bonus greater than 2\euro ~per MWh, while in the non-uniform case, only consumers with ranking lower than 0.2 can claim this level of reward. This highlights the fact that the retailers does not need to propose a reward as huge as in the previous case since the reduction effort is mostly endorsed by users with electric heating, now more compliant to lower their consumption.
%and this still ensures to them a mean reward around 33\euro ~ over the three years. 

\begin{figure}[!ht]
    \centering
    \begin{subfigure}{.47\textwidth}
        \centering
        \includegraphics[width=\linewidth,clip=true,trim=.2cm 0cm 1.3cm 1.2cm]{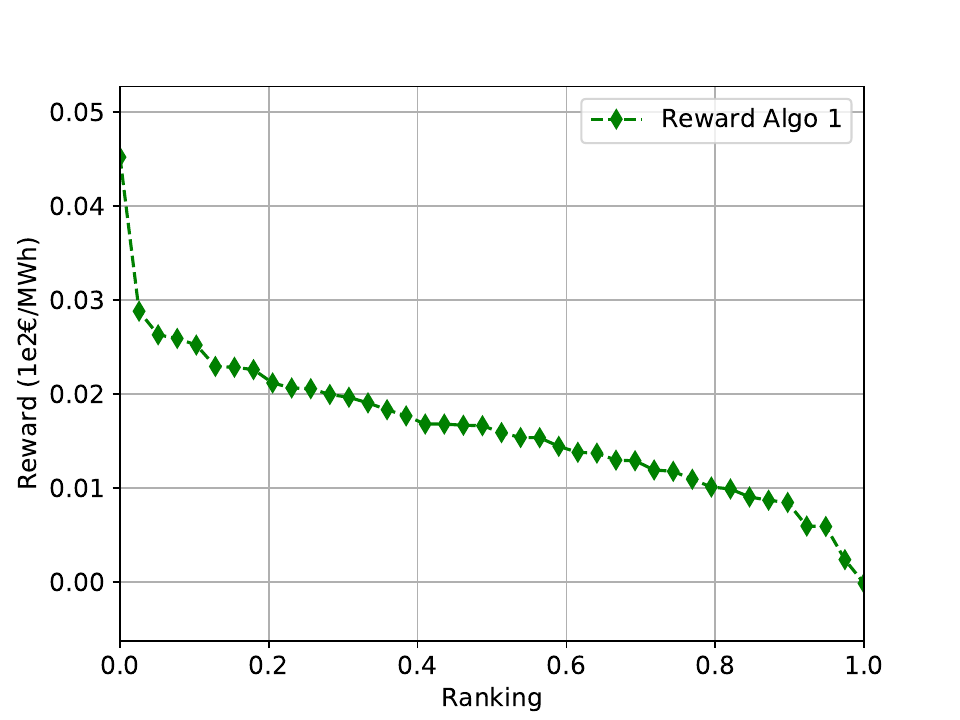}  
        \caption{Unitary bonus function found by~\Cref{algo::CMA}.~\\~}
        \label{fig::retailer_4k_2}
    \end{subfigure}
    \hspace{0.2cm}
    \begin{subfigure}{.47\textwidth}
        \centering
        \includegraphics[width=\linewidth,clip=true,trim=.2cm 0cm 1.3cm 1.2cm]{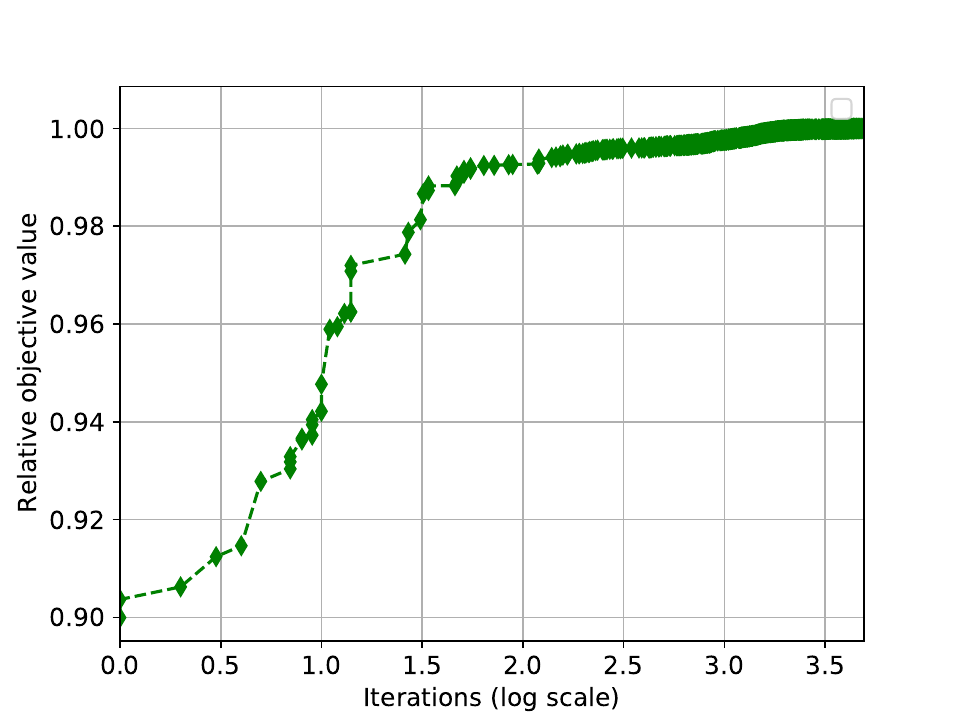}  
        \caption{Evolution of the relative objective value along the
iterations.}
        \label{fig::convergence_4k_2}
    \end{subfigure}
    \begin{subfigure}{\textwidth}
        \centering
        \includegraphics[width=\linewidth,clip=true,trim=.0cm 0cm 0cm 0cm]{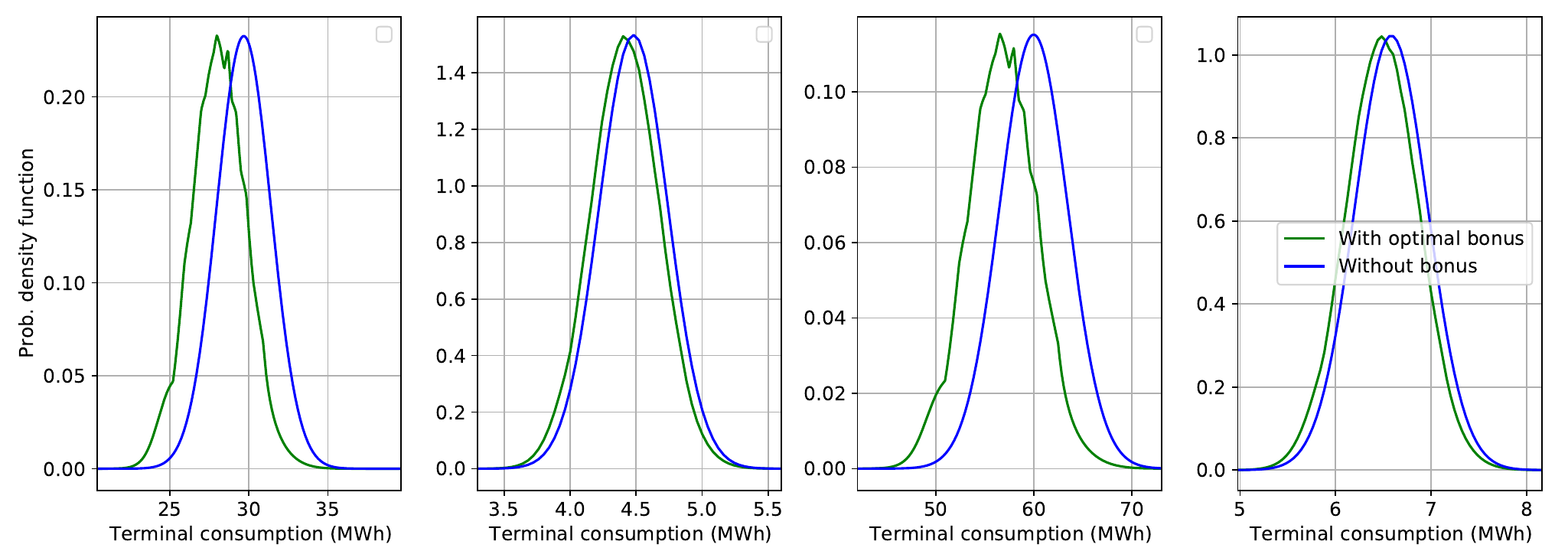}  
        \caption{Distribution of the cumulated consumption over the period $[0,T]$ for the four sub-populations}
        \label{fig::pdf_4k_2}
    \end{subfigure}
    \caption{Numerical results for the four populations with different price elasticity.}
    \label{fig::nonscalable}
\end{figure}

% ===================================
%\clearpage
\section{Extensions}\label{sec::extensions}

We propose in this section several extensions to fit with more general settings.

\paragraph{Energy consumption with common-noise} The add of common-noise is not rare in the modeling of electricity consumption. But in this present case, it does not impact the retailer problem. Intuitively, as the reward is determined by the ranking of the agents, an identical perturbation of the consumption will not modify the rankings, and so the effort made by the agents is independent of the common-noise.

Let us prove this intuitive behavior. To this purpose, we fix a sub-population $k\in[K]$, and suppose that the dynamics is now described as:
\begin{equation}\label{eq::dyn_commonnoise}
d X^a_k(t) = a_k(t)dt + \sigma_k dW_k(t) + \sigma^0 dW^0(t),\quad X_k(0) = \xnom_k\enspace.
\end{equation}
\begin{prop}[Translation invariance of the effort]
    Let $R_k$ be the total reward for sub-population $k$ (satisfying~\Cref{hypo::reward}) and $\mu_k$ be the equilibrium distribution under $R_k$ and without common-noise (given by~\eqref{eq:q_nu}-\eqref{eq:q_mu}). Then
    \begin{equation}\label{eq::mu0}
        \mu^0_k := x\mapsto \mu_k(x - \sigma^0 W^0(T) )
    \end{equation}
    is a (random) equilibrium distribution under $R_k$ and dynamics~\eqref{eq::dyn_commonnoise}. 
\end{prop}
\begin{proof}
    For all $x,q \in \R$ and $\mu\in\mathcal{P}(\R)$, we have:
    $$R_{k,\mu_k}(x+q) = B_k(F_{\mu_k}(x+q)) - p(x+q) = B_k(F_{\mu_k(\cdot +q)}(x)) - p(x+q) = R_{k,\mu_k(\cdot+q)}(x) - pq\enspace.$$
    Therefore, according to the expression of the optimal effort in~\eqref{eq::optimal_effort}
    $$
    u_k(t,x+q,\mu) = q^{-1} u_k(t,x,\mu(\cdot+q))
    $$
    and $a_k(t,x+q;\mu_k) = a_k(t,x;\mu(\cdot+q))$. Therefore, the drift is translation invariant, and the results of~\cite{Lacker_2015} apply: $\mu^0$ defined in~\eqref{eq::mu0} is an equilibrium distribution for the dynamics with common-noise. 
\end{proof}

In contrast with the purely rank-based case, total rewards satisfying~\Cref{hypo::reward} are not translation invariant. Nonetheless, the drift obtained through the optimal effort is translation invariant, enabling to use the results of~\cite{Lacker_2015}. For a common-noise $W^0$ such that $\Esp[W^0(\cdot)]=0$, maximizing the (expected version of the) profit, defined in~\eqref{eq::master_problem}, will boil down to the same problem, and so will lead to the same optimal unitary reward.

\paragraph{General reward $R(x,r)$} We consider here a more general form of reward, coupling the cumulated consumption and the ranking. Therefore, Assumption~\ref{hypo::reward} is no longer satisfied and the equilibrium cannot be explicitly computed with~\Cref{prop::optimal_reward}. Instead, one can used fixed-point resolution techniques to compute the equilibrium.To this purpose, let us denote by $W_1(f_1,f_2)$ the $1$-Wasserstein metric for distribution $f_1,f_2 \in \mathcal{P}_1(\R) = \{\mu\in\mathcal{P}(\R):\,\int_{\R} |x|d\mu(x) <\infty\}$. \Cref{algo::equilibrium} follows the standard way to numerically compute mean-field Nash equilibria -- see~\cite{Achdou_2020} -- by iteratively updating the distribution using the best response operator. Here, the operator is explicitly given by~\eqref{eq::Phi_operator}, which still applies for general forms of reward function, see~\cite{Bayraktar_2019}.
\begin{algorithm}[!ht]
\caption{\textsc{Fixed-point Resolution}}\label{algo::equilibrium}
\begin{algorithmic}
\Require ~\\
\begin{itemize}
\item[-] initial p.d.f. $f_{\mu_k^{(0)}}$ of cluster $k$, 
\item[-] error tolerance $\varepsilon$, 
\item[-] iteration maximum $n_{max}$,
\item[-] sequence of damping coefficients $\{l_i\}_{i\in\mathbb{N}}$.
\end{itemize}
%\State $\hdots \gets \hdots$\Comment{This is a comment}
\State $d,\,i \gets 2\varepsilon,\,0$
\While{$d \geq \varepsilon$ or $n\leq n_{max}$}
  \State $f_{\mu_k^{(i+1/2)}} \gets \Phi_k(f_{\mu_k^{(i)}})$\Comment{Best-response map defined in~\Cref{def::Phi}}
    \State $f_{\mu_k^{(i+1)}} \gets l_i f_{\mu_k^{(i+1/2)}}+ (1-l_i)f_{\mu_k^{(i)}}$ \Comment{damping $l_i$}
    \State $d \gets W_1\left(f_{\mu_k^{(i)}},f_{\mu_k^{(i+1)}}\right)$ \Comment{distance between two iterates}
    \State $i \gets i+1$
\EndWhile
\end{algorithmic}
\end{algorithm}

Instead of Picard iterates ($l_i=1$), a decreasing damping $l_i = \left(\frac{1}{i+1}\right)^p$, $p\in \mathcal{N}$ can be used. The latter sequence of inertial parameters defines iterates of Krasnoselskii-Mann type, which has been proved to converge for pseudo-contractive map in Hilbert space, see~\cite{Rafiq_2007}. Such a damping has been used for example to solve Linear-Quadratic mean-field control problems in~\cite{Grammatico_2016}.

We then show that the uniqueness of the reward function is no longer true in the general setting, and there exists a family of equivalent reward function, going from purely rank-based rewards to purely consumption-based reward ones:
\begin{prop}[Invariance]\label{theorem::invariance}
Let $R^*(x,r)$ be an optimal reward function for the following problem 

\begin{equation}\label{eq::invariance_problem}
\max_{R(x,r)}  \left\{-\kappa(m_\mu) - \int_{\R} R_\mu(x)f_\mu(x) dx\; \left| \; \begin{aligned}
    &\mu = \epsilon(R)\\
    &V(R,\mu) \geq \Vpi
\end{aligned} \right.\right\}
\end{equation}
This equilibrium distribution obtained with $R^*$ is denoted by $\mu^*$. Then, 
\begin{itemize}
    \item [(i)]
the purely rank-based reward function $\hat{B}:r\mapsto R^*(q_{\mu^*}(r),r)$ is also an optimal reward,
\item[(ii)]
the reward function $\hat{R}:x\mapsto R^*(x, F_{\mu^*}(x))$ is also an optimal reward.
\end{itemize}
\end{prop}
\begin{proof}
    By definition, the two reward functions $\hat{B}$ and $\hat{R}$ also satisfy the characterization of the equilibrium~\eqref{eq::charac_eq} with $\mu_k=\mu^*$. Therefore, under these rewards, agents reach the same equilibrium as with $R$, and their utility is identical. Morover, the objective in~\eqref{eq::invariance_problem}.
\end{proof}
In practice, \Cref{theorem::invariance} has very useful implications. It states that complicated reward policies simplify into simple rules. 
The first item shows that we can construct a purely \emph{competitive} game in the sense that the consumers receives incentives only through their rank. The second item shows that we can construct a \emph{decentralized} reward since the incentive of each customer only depends on their own consumption.
Note that this notion of invariance applies at the equilibrium, and the equivalence of the reward is no longer true outside the equilibrium.%\\

%\textit{Numerical experiments.}
%Note that the optimal distribution induced by this reward is not depicted. In fact, as shown in~\Cref{theorem::invariance}, the same optimal equilibrium as in~\Cref{fig::cdf_1k}.

% -----
\paragraph{Time-dependent effort cost}
In the context of the ecological transition, the consumers are more willing to contribute to the energy reduction, and therefore the effort cost $c$ can be viewed as a time dependent parameter, modeling the change of customers' behavior.

In this case, with a cost profile $c_k(t)$, $t\in[0,T]$ for each cluster $k$, the consumer's problem becomes 
\begin{equation}\label{eq::consumer_problem_time_dependent}
    V_k(R,\mu_k) := \sup_a \mathbb{E} \left[ R_{\mu_k}(X^a_{k}(T)) - \int_0^T c_k(t) a^2_k(t) dt \right]\enspace.
\end{equation}
As a direct extension of~\cite{Bayraktar_2016}, we have the following existence result:
\begin{theorem}
Assume that the cost profiles are bounded such that there exist $ (\underline{c}_k,\overline{c}_k)$ verifying for all $t\leq T$
$$0<\underline{c}_k\leq c_k(t)\leq\overline{c}_k\enspace.$$
Then, there exists at least one equilibrium.
\end{theorem}
Nonetheless, there is no more explicit formula (even for the best response of the agents) in presence of time-varying cost of effort, as the Schrödinger bridge method requires a quadratic cost of effort that is constant over time. To illustrate the behavior of the agents with a time-dependent cost of effort, we draw in~\Cref{fig::traj_4k_1_nonlinear} the trajectories of the same 20 consumers as in~\Cref{fig::traj_4k_1} obtained with the incentive depicted in~\Cref{fig::retailer_4k_1} and a cost of effort $c_k(t) = 24 - 1.5t$ \euro/MWh. As expected, the energy savings are greater than in the previous case (the cumulated consumption at the end of the horizon is now around $27.6$MWh whereas it was around $28.5$MWh with $c_k(t) = 24$\euro/MWh.
\begin{figure}[!ht]
    \centering
  \includegraphics[width=0.55\linewidth,clip=true,trim=.2cm 0cm 1.3cm 1.2cm]{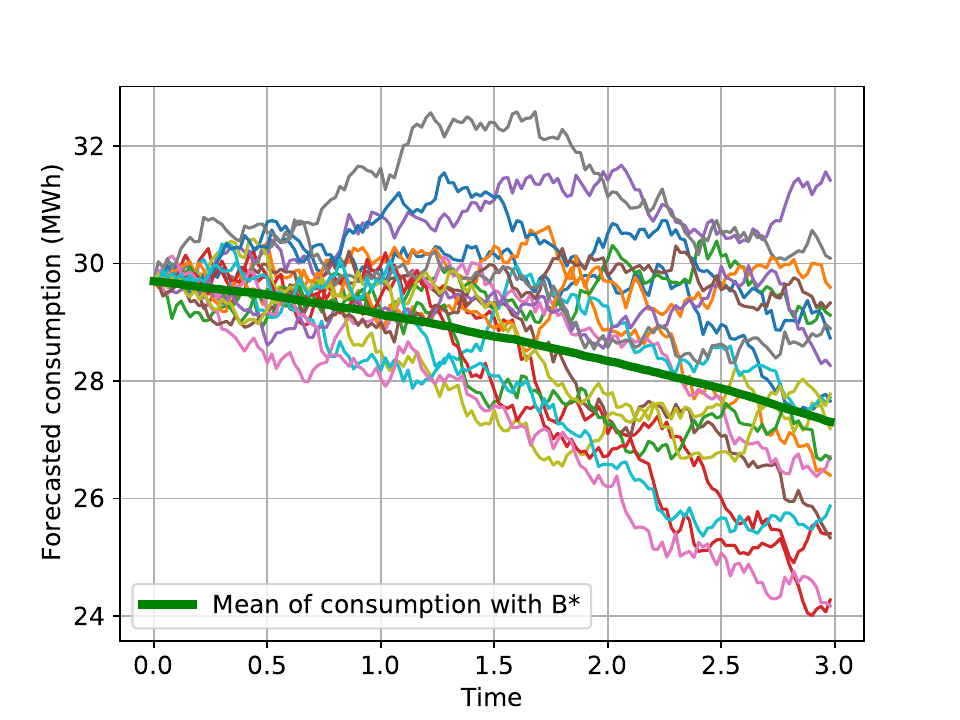}  
  \caption{Deviation of the consumption from the no-bnous case.\\Trajectories of 20 consumers from sub-population 1 obtained with the optimal control from the mean-field approximation and a time-dependent cost of effort.}
  \label{fig::traj_4k_1_nonlinear}
\end{figure}

% =========================
\section{Conclusion}
In this work, we study a Principal-Agent mean-field game where the incentive designed by the principal is based on the ranking of each agent, initiating a competition between them. This specific framework allows us to derive explicit formula for the (unique) mean-field Nash equilibrium for the agents' problem. Incorporating this characterization in the principal profit maximization problem, we prove in the homogeneous setting that the optimal reward can be obtained by solving a convex reformulation of the problem in the distribution space. We exploit the optimality conditions of the latter to then get the optimal reward through a fixed-point equation. In the general case, we show that the problem can be recast as a finite-dimensional maximization over a box, which can be efficiently solved by numerical algorithms. In future work, we plan to explore the convergence of the model restricted to piecewise linear rewards towards the infinite-dimensional case when the discretization step goes to zero.

We apply the results to electricity markets where a provider aims at designing a reward for its consumers portfolio in order to incentivize them to energy sobriety. We construct realistic instances for the French market of Energy Saving Certificates, and numerically observe that the rank-based rewards can constitute efficient mechanisms to make substantial energy reduction, while staying sufficiently simple to be easily grasped by the consumers.

\section*{Acknowledgments}
We thank the reviewers for their fruitful comments, which contributed to an improved version of the paper.

\bibliographystyle{siamplain}
\bibliography{bibliography}

% ===============================================================
%\clearpage
\appendix
\section{Proofs}
In this section, we collect several results and proofs.
\begin{lemma}\label{prop::convo}
\begin{equation}\label{eq::shift}
\fnom_k(x) \exp\left(\tau x\right)
=\exp\left(\tau \xnom_k +\frac{1}{2}\tau^2\sigma^2_k T\right)\varphi\left(x;\xnom_k+\tau \sigma_k^2 T,\sigma_k\sqrt{T}\right)\enspace.
\end{equation}
\end{lemma}
\proof
$$
\begin{aligned}
    \fnom(x) \exp\left(\tau x\right)
    &= \frac{1}{\sigma\sqrt{T}\sqrt{2\pi}}\exp\left(-\frac{(x - \xnom)^2-2\tau \sigma^2 T x}{2\sigma^2 T}\right)\\
    &= \frac{1}{\sigma\sqrt{T}\sqrt{2\pi}}\exp\left(-\frac{(x - [\xnom+\tau \sigma^2 T])^2}{2\sigma^2 T} + \tau \xnom + \frac{1}{2}\tau^2\sigma^2 T\right)
\end{aligned}
$$
\endproof

\begin{lemma}[Set of attainable equilibria]\label{prop::map_reward}
\begin{enumerate}[label=(\roman*)]
    \item For a given cluster $k$, the set of equilibria attainable by an additional reward function $B$ is given by
    $$\mathcal{E}_k = \{\mu\in\mathcal{P}^+(\R): 2c_k\sigma_k^2\ln \zeta_{k,\mu_k} (q_{\mu_k}(r)) + pq_{\mu_k}(r) \text{ is bounded and decreasing} \}\enspace,$$
    with $\zeta_{k,\mu} := f_{\mu}/\fnom_k$.

    \item If $\mu_k \in \mathcal{E}_k$, then 
    $$\epsilon_k^{-1}(\mu_k) = \left\{2c_k\sigma_k^2\ln \zeta_{k,\mu_k} (q_{\mu_k}(r)) + pq_{\mu_k}(r) + C_k:C_k\in\R\right\}$$
    
    \item Suppose that additional reservation ``utility" constraint $V_k(R,\mu_k)\geq \Vpi_{k}+\tau \xnom_k$ and budget constraint $\int_0^1 B(r) dr \leq K$, then the constant $C_k$ in (ii) is restricted to $$\Vpi_{k}+\tau \xnom_k\leq C_k\leq K - 2c_k\sigma_k^2\int_0^1\ln \zeta_{k,\mu_k}(q_{\mu_k}(r))dr - p m_{\mu_k}\enspace.$$
    In particular, such a $C_k$ exists if and only if 
    $$2c_k\sigma_k^2\int_0^1 \ln\zeta_{k,\mu_k}(q_{\mu_k}(r))dr - p m_{\mu_k} \leq K - \Vpi_{k}-\tau \xnom_k\enspace.$$
\end{enumerate}
\end{lemma}
\proof
Items (i) and (iii) directly comes from~\cite{Bayraktar_2019}. For (ii), the condition of~\Cref{lemma:condition_quantile} is verified:
$$
\int_0^r \exp\left(-\frac{R_{\mu}(q_{\mu}(z))}{2c\sigma^2}\right) dz
= \int_0^r \left(\zeta_\mu (q_\mu(r))\right)^{-1}dz
= \int_{-\infty}^{q_\mu(r)} \fnom(z)dz
\enspace.$$
As the uniqueness is concerned, suppose that $B$ and $B'$ lead to the same distribution $\mu$ with $p\neq 0$. Then, $B$ and $B'$ lead to the same distribution $\nu$ with $p=0$, see~\Cref{theorem:q_mu}. Therefore, as shown in~\cite{Bayraktar_2019}, $B$ and $B'$ are equal up to a constant. 
\endproof

\subsection*{Proof of~\Cref{theorem:q_mu}}\label{app::proof_shift}
We give here the proof for a given class and, for simplicity, we omit the dependence in $k$.\\

\textit{Characterization of an equilibrium.} First, suppose that $\nu$ is an equilibrium distribution for the case $p=0$. Let $\gamma \in \R$ whose value will be determined later. By definition of $f_\nu$ (see \eqref{eq::Phi_operator}), we get
$$
\begin{aligned}
&\int_0^r \exp\left(-\frac{B(z) - p(q_\nu(z)+\gamma)}{2c\sigma^2}\right) dz =\int_{-\infty}^{q_\nu(r)} \exp\left(-\frac{B(F_\nu(x))}{2c\sigma^2} + \frac{p}{2c\sigma^2}(x+\gamma)\right) f_\nu(x)dx\\
&= \frac{e^{\frac{p}{2c\sigma^2}\gamma}}{\gamma(\nu)}\int_{-\infty}^{q_\nu(r)} \exp\left(-\frac{B(F_\nu(x))}{2c\sigma^2} + \frac{p}{2c\sigma^2}x\right) \fnom(x)\exp\left(\frac{B(F_\nu(x))}{2c\sigma^2}\right)dx.
\end{aligned}
$$
Using \eqref{eq::shift} with $\tau = \frac{p}{2c\sigma^2}$ and the change of variables $u = \frac{x - (\xnom+\frac{pT}{2c})}{\sigma\sqrt{T}}$, we deduce
$$
\begin{aligned}
&\int_0^r \exp\left(-\tfrac{B(z) - p(q_\nu(z)+\gamma)}{2c\sigma^2}\right) dz = \frac{1}{\gamma(\nu)}e^{\frac{1}{2c\sigma^2}\left(\gamma + p \xnom + \frac{Tp^2}{4c}\right)}\int_{-\infty}^{q_\nu(r)} \varphi\left(x;\xnom+\frac{pT}{2c},\sigma\sqrt{T}\right)dx \\
&= \frac{1}{\gamma(\nu)\sqrt{2\pi}}e^{\frac{1}{2c\sigma^2}\left(\gamma + p \xnom + \frac{Tp^2}{4c}\right)}\int_{-\infty}^{\frac{q_\nu(r)-(\xnom+\frac{pT}{2c})}{\sigma\sqrt{T}}} \exp\left(-\frac{u^2}{2}\right)du \\
&= \frac{1}{\gamma(\nu)}e^{\frac{1}{2c\sigma^2}\left(\gamma + p \xnom + \frac{Tp^2}{4c}\right)}N\left(\frac{q_\nu(r)-(\xnom+\frac{pT}{2c})}{\sigma\sqrt{T}}\right)\enspace.
\end{aligned}
$$
Therefore, taking $\gamma = -\frac{pT}{2c}$, we end up with 
$$N\left(\frac{\left[q_\nu(r)-\frac{pT}{2c}\right]-\xnom}{\sigma\sqrt{T}}\right) = \frac{\int_0^r \exp\left(-\frac{B(z) - p\left[q_\nu(z)-\frac{pT}{2c}\right]}{2c\sigma^2}\right) dz}{\int_0^1 \exp\left(-\frac{B(z) - p\left[q_\nu(z)-\frac{pT}{2c}\right]}{2c\sigma^2}\right) dz} \enspace .$$
By setting $q_{\mu}(r) = q_\nu(r)-\frac{pT}{2c}$, we recover the characterization of an equilibrium (see~\Cref{lemma:condition_quantile}).

Conversely, suppose now that $\mu$ is the equilibrium for $p\in\R$. Then, following the same steps, 
$$N\left(\frac{\left[q_\mu(r)+\frac{pT}{2c}\right]-\xnom}{\sigma\sqrt{T}}\right) = \frac{\int_0^r \exp\left(-\frac{B(z) }{2c\sigma^2}\right) dz}{\int_0^1 \exp\left(-\frac{B(z)}{2c\sigma^2}\right) dz} \enspace .$$
The distribution $\nu$ defined as $q_\nu(r) = q_\mu(r) + \frac{pT}{2c}$ is a valid equilibrium.\\

\textit{Uniqueness of the equilibrium.} Suppose that there exist two distinct equilibrium distributions $\mu$ and $\mu'$ such that $q_\mu \neq q_{\mu'}$. Then by the above proof, we derive
the existence of two distinct equilibrium distributions $\nu$ and $\nu'$  for the case $p=0$ satisfying $q_\nu \neq q_{\nu'}$. We get a contradiction by the uniqueness of the equilibrium for purely rank-based rewards.

% -----
\subsection*{Proof of~\Cref{lemma::optimal_reward}}\label{app::optimal_reward}
We apply the KKT conditions on~\eqref{eq:relaxed_problem}: for $\mu^*$-almost every $x$ in $\R$,
\begin{equation*}
\left\{
\begin{aligned}
    &0 = x \kappa'(m_{\mu^*}) +2c\sigma^2\ln\left(\frac{f_{\mu^*}(x)}{\fnom(x)}\right) + \lambda,\\
    &\int_{-\infty}^{+\infty} f_{\mu^*}(y) dy = 1
    \end{aligned}
    \quad,\lambda\in \R\right.
\end{equation*}
From which we can deduce that
$
f_{\mu^*}(x) = \fnom(x) \exp\left(- \frac{x\kappa'(m_{\mu^*}) + \lambda}{2c\sigma^2}\right)
$. The Lagrange multiplier $\lambda$ is then computed using the normalization condition on $f_{\mu^*}$.

% -----
\subsection*{Proof of~\Cref{prop::optimal_reward}}
Integrating~\eqref{eq::fixed_point_fmu} gives us

$$
\begin{aligned}
m_\mu = \int_{-\infty}^{+\infty} y f_\mu(y) dy &= \frac{1}{\alpha(\mu)}\int_{-\infty}^{+\infty} y \fnom(y) \exp\left(-y \frac{\kappa'(m_\mu)}{2c\sigma^2}\right) dy\\
&=\int_{-\infty}^{+\infty} y \phi\left(y;\xnom-\frac{T\kappa'(m_\mu)}{2c},\sigma\sqrt{T}\right)dy\\
&= \xnom -\frac{T\kappa'(m_\mu)}{2c} = \xpi + \frac{T}{2c}\delta(m_\mu)\enspace,
\end{aligned}
$$
where we use~\Cref{prop::convo} between the two first lines in order to recover a gaussian process.

We can now recover the reward:
$$
\begin{aligned}
B^*(r) &=\Vpi+\tau \xnom + 2c\sigma^2 \ln\left(\zeta_{\mu^*}(q_{\mu^*}(r))\right) + p q_{\mu^*}(r)\\
&= \Vpi+\tau \xnom + q_{\mu^*}(r) \left[p - \kappa'(m_{\mu^*})\right] - 2c\sigma^2 \ln\left(\int_{-\infty}^{+\infty} \fnom(y) \exp\left(-y \frac{\kappa'(m_{\mu^*})}{2c\sigma^2}\right)dy\right)\\
&=\Vpi +\tau \xnom+ \frac{c}{T}\left[(\xnom)^2 - m^2\right] + q_{\mu^*}(r) \delta(m_{\mu^*}) \\
&=\tau \xnom + \frac{c}{T}\left[(\xpi)^2 - m^2\right] + q_{\mu^*}(r) \delta(m_{\mu^*}) \enspace,
\end{aligned}
$$
where we use~\Cref{prop::convo} to get the value of the integral.
From the definition of the provider objective,
$$\begin{aligned}
\pi & = pm - \kappa(m) - \int_0^1 B^*(r)dr\\
&= pm - \kappa(m) - \frac{c}{T}\left[(\xpi)^2 - m^2\right] - m \left[p - \kappa'(m)\right] - \tau \xnom\\
&= m \kappa'(m) - \kappa(m) +\left(\frac{\xpi + m}{2}\right)\frac{2c}{T}\left(m - \xpi\right)- \tau \xnom\\
&= m \kappa'(m) - \kappa(m) + \left(\frac{\xpi + m}{2}\right) \delta(m)- \tau \xnom \enspace.
\end{aligned}$$

% -----------------------------
\subsection*{Proof of~\Cref{theorem::invariance}}
\begin{itemize}
    \item[(i)] By construction, the reward $\tilde{B}$ is also bounded and decreasing. Then, the cost induced by the additional reward is the same with $R^*$ and $\hat{B}$:
    $$\int_{-\infty}^{+\infty} R^*_{\mu^*}(x)f_{\mu^*}(x)dx = \int_0^1 \hat{B}(r)dr\enspace.$$
    Finally, $\mu^*$ is also an equilibrium for the reward $\hat{B}$:
    $$\frac{1}{\hat{\gamma}(\mu^*)}\fnom(x)\exp\left(\frac{\hat{B}(F_{\mu^*}(x))}{2c\sigma^2}\right) = \frac{1}{\gamma^*(\mu^*)}\fnom(x)\exp\left(\frac{R^*_{\mu^*}(x)}{2c\sigma^2}\right) = f_{\mu^*}\enspace,$$
    where $\hat{\gamma}$ and $\gamma^*$ are computed respectively with $\hat{B}$ and $R^*$. The last equality comes from the characterization of an equilibrium.
    Therefore, the reward function $\hat{B}$ satisfies the constraints and produces the same objective value as $R^*$. It is also optimal.
    \item[(ii)] The proof follows the same ideas as at the previous item.
\end{itemize}

\end{document}